\documentclass[10pt]{amsart}
\usepackage{amssymb, enumitem}
\usepackage[all]{xy}
\usepackage{hyperref, aliascnt}
\usepackage{mathtools}
\usepackage[english]{babel}
\usepackage{enumitem}
\usepackage{comment}
\usepackage[nameinlink]{cleveref}
\usepackage{amsmath,amsthm,amsfonts,amscd, latexsym,mathrsfs,stmaryrd,bbm}

\def\today{\number\day\space\ifcase\month\or   January\or February\or
	March\or April\or May\or June\or   July\or August\or September\or
	October\or November\or December\fi\   \number\year}

\newtheorem{lma}{Lemma}[section]

\newaliascnt{thmCt}{lma}
\newtheorem{thm}[thmCt]{Theorem}
\aliascntresetthe{thmCt}

\newaliascnt{corCt}{lma}
\newtheorem{cor}[corCt]{Corollary}
\aliascntresetthe{corCt}

\newaliascnt{constrCt}{lma}
\newtheorem{constr}[constrCt]{Construction}
\aliascntresetthe{constrCt}

\newaliascnt{propCt}{lma}
\newtheorem{prop}[propCt]{Proposition}
\aliascntresetthe{propCt}

\newtheorem*{thm*}{Theorem}
\newtheorem*{cor*}{Corollary}
\newtheorem*{prop*}{Proposition}

\theoremstyle{definition}
\newtheorem*{def*}{Definition}

\theoremstyle{plain}
\newtheorem{thmintro}{Theorem}

\newcounter{theoremintro}

\newaliascnt{corintroCt}{thmintro}
\newtheorem{corintro}[corintroCt]{Corollary}
\aliascntresetthe{corintroCt}

\theoremstyle{definition}
\newtheorem{pbmintro}[theoremintro]{Problem}

\newaliascnt{pgrCt}{lma}

\aliascntresetthe{pgrCt}

\newaliascnt{dfCt}{lma}

\theoremstyle{definition}
\newtheorem{df}[dfCt]{Definition}
\aliascntresetthe{dfCt}

\newaliascnt{remCt}{lma}
\newtheorem{rem}[remCt]{Remark}
\aliascntresetthe{remCt}

\newaliascnt{remsCt}{lma}

\aliascntresetthe{remsCt}

\newaliascnt{egCt}{lma}

\aliascntresetthe{egCt}

\newaliascnt{egsCt}{lma}

\aliascntresetthe{egsCt}

\newaliascnt{qstCt}{lma}

\aliascntresetthe{qstCt}

\newaliascnt{pbmCt}{lma}

\aliascntresetthe{pbmCt}

\newaliascnt{notaCt}{lma}
\newtheorem{nota}[notaCt]{Notation}
\aliascntresetthe{notaCt}

\newcommand{\beq}{\begin{equation}}
	\newcommand{\eeq}{\end{equation}}
\newcommand{\beqa}{\begin{eqnarray*}}
	\newcommand{\eeqa}{\end{eqnarray*}}
\newcommand{\bal}{\begin{align*}}
	\newcommand{\eal}{\end{align*}}
\newcommand{\bi}{\begin{itemize}}
	\newcommand{\ei}{\end{itemize}}
\newcommand{\be}{\begin{enumerate}}
	\newcommand{\ee}{\end{enumerate}}

\newcommand{\N}{{\mathbb{N}}}

\newcommand{\supp}{{\mathrm{supp}}}

\newcommand{\ca}{$C^*$-algebra}

\newcommand{\abs}[1]{\left|#1\right|}

\usepackage[x11names]{xcolor}

\title{Essential freeness, allostery and $\mathcal{Z}$-stability of crossed
	products}

\date{\today}

\thanks{
	The first named author was partially supported by the Swedish Research Council Grant 2021-04561.
	The rest of the authors were funded by the Deutsche Forschungsgemeinschaft (DFG, German Research Foundation) – Project-ID 427320536 – SFB 1442, as well as by Germany’s Excellence Strategy – EXC 2044 – 390685587, Mathematics M\"{u}nster – Dynamics – Geometry – Structure; and ERC Advanced Grant 834267 -- AMAREC}

\author[E.~Gardella]{Eusebio Gardella}
\address{Eusebio Gardella,
	Department of Mathematical Sciences,
	Chalmers University of Technology and University of Gothenburg,
	Gothenburg SE-412 96, Sweden}
\email{gardella@chalmers.se}

\author[S.~Geffen]{Shirly Geffen}
\address{Shirly Geffen,
	Mathematisches Institut,
	University of M{\"u}nster, 
	Einsteinstr.\ 62, 
	48149 M{\"u}nster, Germany}
\email{sgeffen@uni-muenster.de}

\author[R.~Gesing]{Rafaela Gesing}
\address{Rafaela Gesing,
	Mathematisches Institut,
	University of M{\"u}nster, 
	Einsteinstr.\ 62, 
	48149 M{\"u}nster, Germany}
\email{rgesing@uni-muenster.de}

\author[G.~Kopsacheilis]{Grigoris Kopsacheilis}
\address{Grigoris Kopsacheilis,
	Mathematisches Institut,
	University of M{\"u}nster, 
	Einsteinstr.\ 62, 
	48149 M{\"u}nster, Germany}
\email{gkopsach@uni-muenster.de}

\author[P.~Naryshkin]{Petr Naryshkin}
\address{Petr Naryshkin,
	Mathematisches Institut,
	University of M{\"u}nster, 
	Einsteinstr.\ 62, 
	48149 M{\"u}nster, Germany}
\email{pnaryshk@uni-muenster.de}

\begin{document}
	
\begin{abstract}
We explore classifiability of crossed products of actions of countable amenable groups on compact, metrizable spaces. 
It is completely understood when such crossed products are simple, separable, unital, nuclear and satisfy the UCT: these properties are equivalent to the combination of minimality and topological freeness, and the challenge in this context is establishing $\mathcal{Z}$-stability. While most of the existing results in this direction assume freeness of the action, there exist numerous natural examples of minimal, topologically free (but not free) actions whose crossed products are classifiable. 
		
In this work, we take the first steps towards a systematic study of $\mathcal{Z}$-stability for crossed products beyond the free case, extending the available machinery around the small boundary property and almost finiteness to a more general setting. Among others, for actions of groups of polynomial growth with the small boundary property, we show that minimality and topological freeness are not just necessary, but also \emph{sufficient} conditions for classifiability of the crossed product. 
		
Our most general results apply to actions that are essentially free, a property weaker than freeness but stronger than topological freeness in the minimal setting. Very recently, M.\ Joseph produced the first examples of minimal actions of amenable groups which are topologically free and not essentially free. While the current machinery does not give any information for his examples, we develop ad-hoc methods to show that his actions have classifiable crossed products.
\end{abstract}

\maketitle

\tableofcontents

\renewcommand*{\thetheoremintro}{\Alph{theoremintro}}
\section{Introduction}

One of the most remarkable achievements in $C^*$-algebra theory in the last decades is the completion of the classification programme initiated by George Elliott over 30 years ago. The outcome is the combination of the work of a large number of mathematicians over several decades (see, for instance, \cite{Phi_classification_2000,GonLinNiu_classification_2015,		EllGonLinNiu_classification_2015, TikWhiWin_quasidiagonality_2017, CasEviTikWhiWin_nuclear_2019}), and can be stated as follows:

\begin{thm*}(Classification Theorem).
Let $A$ and $B$ be simple, separable, unital, nuclear, $\mathcal{Z}$-stable \ca s satisfying the UCT. Then $A$ and $B$ are isomorphic if and only if their Elliott invariants are isomorphic.
\end{thm*}

We refer the reader to Winter's 2018 ICM address \cite{Win_ICM} for an overview and historical account of the developments in this direction, as well as for a detailed explanation of the assumptions in the theorem. There is also a very recent alternative approach to prove the classification theorem \cite{CarGabSchTikWhi_classifying_2023}, relying on classification results for maps between suitable classes of $C^*$-algebras. We refer the reader to White's 2022 ICM address \cite{Whi_ICM} for more on this approach.

With such a powerful classification theorem at our disposal, it becomes an imperative task to identify interesting classes of $C^*$-algebras to which it can be applied. One of the most natural families of $C^*$-algebras arises from topological dynamics via the crossed product construction.

In recent years, a significant amount of work has been done to identify dynamical criteria for an action $G\curvearrowright X$ of a countable discrete group on a compact metric space that ensure that the associated crossed product $C(X)\rtimes G$ satisfies the assumptions of the classification theorem\footnote{Since we will work exclusively with amenable actions, the choice of the crossed product completion here is irrelevant.}. Unitality and separability of $C(X)\rtimes G$ are automatic, while nuclearity of $C(X)\rtimes G$ is equivalent to amenability of $G\curvearrowright X$. Moreover, if $G\curvearrowright X$ is amenable, then $C(X)\rtimes G$ automatically satisfies the UCT by results of Tu \cite{Tu99}, and it is simple if and only if $G\curvearrowright X$ is minimal and topologically free by results of Archbold and Spielberg \cite{ArcSpi94}. In particular, and we want to stress this, amenability, minimality, and topological freeness are \emph{necessary} conditions for classifiability of $C(X)\rtimes G$, and it remains to decide when $C(X)\rtimes G$ is $\mathcal{Z}$-stable. 

Like most works on the subject, we will restrict our attention to the case where $G$ is itself amenable (which, for amenable actions, is equivalent to the crossed product being stably finite), and refer to \cite{GarGefKraNar_classifiability_2022} for recent progress in the nonamenable case. For the convenience of the reader, and since the subtle differences between these notions play a crucial role in our work, we recall the following:
	
\begin{def*}
Let $G$ be a discrete group, and let $X$ be a topological space. We say that an action $G\curvearrowright X$ is:

\begin{enumerate}
\item \emph{free}, if for every $g\in G\setminus\{1\}$, we have $\{x\in X\colon g\cdot x=x\}=\emptyset$;

\item \emph{essentially free}, if for every $g\in G\setminus\{1\}$, we have $\mu\big(\{x\in X\colon g\cdot x=x\}\big)=0$ for every $G$-invariant Borel probability measure $\mu$ on $X$;

\item \emph{topologically free}, if for every $g\in G\setminus\{1\}$, the interior of $\{x\in X\colon g\cdot x=x\}$ is empty. 
\end{enumerate}

\end{def*}
	
It is clear that freeness implies both essential freeness and topological freeness. For minimal actions of amenable groups, essential freeness also implies topological freeness\footnote{This follows from the fact that if the action is minimal, then no invariant probability measure is zero on a nonempty open set; see \cite[Lemma~2.2]{Jos21}.}. In other words, for the actions we are concerned with in this work, we have (1) $\Rightarrow$ (2) $\Rightarrow$ (3).
	
The distinctness of the three notions defined above turns out to depend on the acting group. Indeed, for minimal actions of \emph{abelian} groups, (2) implies (1): any topologically free action is automatically free\footnote{To prove this, one shows that the fixed point set of any group element is closed, and also $G$-invariant if the group $G$ is abelian.}. However, already the infinite dihedral group $\mathbb{Z}\rtimes\mathbb{Z}_2$ admits a natural essentially free action on the unit circle that is not free\footnote{For example, one can let the copy of $\mathbb{Z}$ act by an irrational rotation and $\mathbb{Z}_2$ act by complex conjugation.}, showing that (2) does not in general imply (1).
	
The difference between (2) and (3) is much more subtle. They are known to coincide for minimal actions of finitely generated groups of polynomial growth and polycyclic groups (see, for example, \cite[Corollary~2.4]{Jos21}), and it was actually open for some time whether there exists a minimal action of an amenable group that satisfies (3) but not (2) -- such actions are called \emph{allosteric} by Joseph \cite{Jos21}. The first allosteric actions of amenable groups were recently constructed by Joseph in \cite{Jos23}, and we refer to the introduction of this paper for more about the significance of the existence of  such actions. Shortly after, a generalization of Joseph's construction was presented by Hirshberg and Wu in \cite{HirWu23}.
	
Returning to the question of classifiability of crossed products, most of the work in this direction has been done under the assumption of freeness, and the aim of the present work is to initiate a systematic study of the question under the less restrictive assumptions of essential freeness or topological freeness. The most general problem can be formulated as follows.

\begin{pbmintro}\label{pbm:intro}
Let $G\curvearrowright X$ be an amenable, topologically free, minimal action of a discrete group on a compact metric space. Find a dynamical characterization of $\mathcal{Z}$-stability for $C(X)\rtimes G$.
\end{pbmintro}
	
In the case of free actions, the modern approach relies on proving that the action is \emph{almost finite}, a notion introduced by Kerr in \cite{Ker_dimension_2020}, which is known to imply that the crossed product is classifiable. By the work of Kerr and Szab{\'o} \cite{KerSza_almost_2020} (always under freeness assumptions), almost finiteness is equivalent to the conjunction of two weaker properties: \emph{almost finiteness in measure} (\autoref{def almost finite in measure}) and \emph{comparison} (\autoref{def sbp}). For free actions, the former is known to be equivalent to the \emph{small boundary property} and frequently holds automatically (for instance, for free actions on finite-dimensional spaces), although not always \cite{GioKer_subshifts_2010}. Thus, the main challenge often lies in proving comparison. 

Although stated for free actions in the literature, the definition of almost finiteness does not imply freeness of the action and, in fact, the property holds for many non-free actions as well. On the other hand, it is easy to see that any action which is almost finite (in measure) must be essentially free (see \autoref{rem:EssFreeAlmostFinite}). Thus, essential freeness arises naturally in the context of the available tools, which suggests that one should try to extend the existing machinery to the essentially free setting. However, even if this can be carried out successfully, new tools would have to be developed to deal with allosteric actions. In this paper we obtain results in both of these situations.
	
For essentially free actions, the first results were obtained by Ortega and Scarparo \cite{OrtSca_almost_2023}, who directly proved almost finiteness for minimal actions of the infinite dihedral group $\mathbb{Z}\rtimes\mathbb{Z}_2$ on Cantor spaces as well as essentially free odometers. The next step has recently been taken by Li and Ma in \cite{LiMa23}, where it is shown that almost finiteness in measure and the small boundary property are equivalent for essentially free actions of locally finite-by-virtually $\mathbb{Z}$ groups. The methods developed there strongly rely on certain permanence properties of dynamical features with respect to extensions of actions, which is why these results are limited to that class of groups.
	
Extending results from the free to the essentially free setting one may not seem to be a significant advance at first sight, but there are in general a number of technical difficulties when dealing with non-free actions. One of the most immediate problems that one encounters is that the restriction of an essentially free action to a subgroup is not necessarily essentially free (unlike the situation for free actions), thus obstructing a number of arguments involving extensions of groups and restrictions of actions.

In the first part of this paper, we draw inspiration from \cite{Ker_dimension_2020,KerSza_almost_2020} and extend some of their main results to essentially free actions. While our arguments largely follow theirs, we do have to carefully handle fixed point sets and estimate their measures with respect to invariant Borel probability measures. Our approach, despite leading to tedious arguments at times, allows us to overcome the limitations regarding the structure of the groups in the results obtained in \cite{LiMa23}.

Our first main result is as follows.

\begin{thmintro}\label{thmintro: AF in measure}
Let $G\curvearrowright X$ be an essentially free action of a discrete amenable group on a compact metrizable zero-dimensional space. Then $G\curvearrowright X$ is almost finite in measure.
\end{thmintro}

As a consequence, we obtain a generalization of \cite[Theorem A]{KerSza_almost_2020}:

\begin{thmintro}\label{thmintro: equivalence}
Let $G\curvearrowright X$ be an essentially free action of a discrete amenable group on a compact metrizable space. Then $G\curvearrowright X$ has the small boundary property if and only if it is almost finite in measure. Moreover, the following are equivalent:

\begin{enumerate}

\item the action is almost finite,

\item the action has the small boundary property and comparison.

\end{enumerate}
\end{thmintro}

Note that \autoref{thmintro: equivalence} does not hold beyond the essentially free setting. The allosteric actions constructed in \cite{Jos23} have the small boundary property and comparison, but are not almost finite (in measure).

Extending other results from free actions to essentially free ones seems to be rather challenging. For example, it is known (see \cite[Theorem B]{KerSza_almost_2020}) that all free actions of a given group on finite-dimensional spaces are almost finite if and only if all its free actions on zero-dimensional spaces are almost finite. The proof of this result crucially relies on the fact that free actions on finite-dimensional spaces have Lindenstrauss' \emph{topological small boundary property}. Since this is expected to fail for essentially free actions, it is very unclear how to proceed. 
Nevertheless, we do prove that such actions still have the usual small boundary property.
	
\begin{thmintro}\label{thmintro: FD SBP}
Let $G \curvearrowright X$ be an essentially free action of a discrete amenable group on a compact metrizable finite-dimensional space. Then it has the small boundary property.
\end{thmintro}
	
One particularly nice class of groups to consider is that of finitely generated groups of polynomial growth. As mentioned above, their minimal actions can never be allosteric, and they are also known to always have comparison by the main result of \cite{Nar22}. Thus, we obtain the following characterization, generalizing the results in \cite{OrtSca_almost_2023} and \cite{LiMa23}. 

\begin{corintro}\label{corintro: polygrowth}
Let $G$ be a finitely generated group of polynomial growth, and let $G  \curvearrowright X$ be an action on a compact metrizable space with the small boundary property. Then $C(X)\rtimes G$ is classifiable if and only if $G \curvearrowright X$ is minimal and topologically free. 
\end{corintro}

By combining \autoref{thmintro: FD SBP} and \autoref{corintro: polygrowth} we recover a recent result of Hirshberg-Wu \cite{HirWu23}: for an action $G\curvearrowright X$ of a finitely generated group with polynomial growth on a finite-dimensional compact space, the crossed product $C(X)\rtimes G$ is classifiable if and only if the action is topologically free and minimal. 

Although some of the results above are stated for \emph{topologically} free actions, the machinery we have developed is only suitable to deal with \emph{essentially} free actions. In other words, none of the results mentioned so far apply to allosteric actions (that is, minimal actions that are topologically free but not essentially free), and in particular not to the examples constructed by Joseph \cite{Jos23}. It was left as an open question in his work to  determine whether the crossed products of his actions are classifiable.

In the last section of this work, we analyze Joseph's construction and develop ad-hoc methods to obtain $\mathcal{Z}$-stability for the crossed products of his actions:

\begin{thmintro}\label{thm:Joseph}
Let $G\curvearrowright X$ be one of the allosteric actions of amenable groups constructed in \cite{Jos23}. Then $C(X)\rtimes G$ is $\mathcal{Z}$-stable and thus classifiable.
\end{thmintro}

Our arguments to prove \Cref{thm:Joseph}, although designed to work specifically for the actions at hand, may very well be the first step towards producing a version of almost finiteness that is suitable for topologically free actions. Indeed, our proof of \Cref{thm:Joseph} relies on the construction of Rokhlin-type towers along the ``free'' direction of the action, that are moreover invariant with respect to the ``non-free'' direction, which is possible thanks to the specific form of the actions considered in Joseph's construction. It is conceivable that a more refined version of our methods can be applied in more general contexts.


	\section{Preliminaries}\label{sec prelim}

	\subsection{F{\o}lner sequences}
	
	Given a group $G$, a finite subset $K\subseteq G$, and $\varepsilon>0$, we say that a finite subset $F\subseteq G$ is $(K,\varepsilon)$-\textit{invariant} if $\frac{\left|K\cdot F\triangle F\right|}{\left|F\right|}<\varepsilon$. Amenability of a group $G$ is equivalent to the existence of a $(K,\varepsilon)$-invariant subset for 
	every $K$ and $\varepsilon$. When $G$ is countable, this is in turn equivalent to the existence of a sequence 
	$(F_n)_{n\in\mathbb{N}}$ of finite subsets $F_n\subseteq G$ such that, for any finite subset $K\subseteq G$ and any $\varepsilon>0$, there exists $n_0\in\N$ so that $F_n$ is $(K,\varepsilon)$-invariant for all $n\geq n_0$. We call such a sequence
	a \textit{F{\o}lner sequence} for $G$.
	
	The following observation is straightforward, and we isolate it for later use.
	
\begin{rem}\label{rmk Folner}
Let $(F_n)_{n\in\mathbb{N}}$ be a F{\o}lner sequence for a countably infinite group $G$.
\begin{enumerate}
\item We have $\lim_{n\to\infty}\left|F_n\right|=\infty$.

\item For any finite subset $K\subseteq G$, the sequence $(F_n\cup K)_{n\in\mathbb{N}}$ is also a F{\o}lner sequence for $G$.
\end{enumerate}
\end{rem}

The following is \cite[Lemma~3.12]{KerSza_almost_2020}.

\begin{lma}\label{lem invariance}
Let $K$ be a finite subset of a group $G$, and let $\delta>0$ be given. Then there exists $\varepsilon>0$ such that, if $F\subseteq G$ is a finite subset which is $(K,\varepsilon)$-invariant, then any $F'\subseteq F$ with $\left|F'\right|\ge(1-\varepsilon)\left|F\right|$ is $(K,\delta)$-invariant.
\end{lma}

\subsection{Banach densities}
Let $G$ be a discrete group acting on a compact metrizable space $X$. The set of $G$-invariant Borel probability measures on $X$ will be denoted by $M_G(X)$. Note that $M_G(X)$ is a weak$^\ast$-compact, convex subset, which is non-empty whenever the acting group $G$ is amenable.

Next, we recall the definitions of Banach densities and some basic facts about them. For the proofs, we refer to \cite[Section 3]{KerSza_almost_2020}. 

\begin{df}
Let $G\curvearrowright X$ be an action of a discrete group $G$ on a topological space $X$. For a nonempty finite subset $F\subseteq G$ and a set $A\subseteq X$, define
\[
\underline{D}_F(A)\coloneqq\inf_{x\in X}\frac{1}{\left|F\right|}\sum_{t\in F}\mathbbm{1}_A(tx)=\inf_{x\in X}\frac{\left|\{t\in F\colon tx\in A\}\right|}{\left|F\right|},\]
and
\[
\overline{D}^F(A)\coloneqq \sup_{x\in X}\frac{1}{\left|F\right|}\sum_{t\in F}\mathbbm{1}_A(tx)=\sup_{x\in X}\frac{\left|\{t\in F\colon tx\in A\}\right|}{\left|F\right|}.\]
We define the \textit{lower} and \textit{upper Banach densities} of $A$, respectively, as

\[
\underline{D}(A)\coloneqq \sup_{F\subseteq G}\underline{D}_F(A), \ \ \ \text{ and }\ \ \ \overline{D}(A)\coloneqq \inf_{F\subseteq G}\overline{D}^F(A),\]
where $F$ ranges over the non-empty finite subsets of $G$ in both cases.
\end{df}

If $G$ is amenable and countable and $(F_n)_{n\in\mathbb{N}}$ is a F{\o}lner sequence in $G$, one has

\begin{equation}\label{eq densities Folner sets}\tag{2.1}
\underline{D}(A)=\lim_{n\to\infty}\underline{D}_{F_n}(A) \ \ \ \text{ and } \ \ \ \overline{D}(A)=\lim_{n\to\infty}\overline{D}^{F_n}(A).
\end{equation}

As a consequence (see \cite[Proposition~3.3]{KerSza_almost_2020}), it follows that for any open subset $U\subseteq X$ and any closed subset $C\subseteq X$, one has
\begin{equation}\label{eq densities suprema infima}\tag{2.2}
\underline{D}(U)=\inf_{\mu\in M_G(X)}\mu(U) \ \ \ \text{ and } \ \ \ \overline{D}(C)=\sup_{\mu\in M_G(X)}\mu(C).
\end{equation}

When $X$ is assumed to be metrizable and $d$ is a compatible metric on it, for $\eta>0$ and a subset $A\subseteq X$, we set $A^{<\eta}\coloneqq\{x\in X\colon d(x,A)<\eta\}$. If $A$ is closed, then it is a consequence of the Portmanteau Theorem that
\begin{equation}\label{eq upper dens lim of nbds}\tag{2.3}
\overline{D}(A) = \lim_{\eta \to 0^+}\overline{D}(A^{<\eta}).
\end{equation}
In particular, if $X$ is a zero dimensional space, then we can find a decreasing sequence $(C_n)_{n\in\mathbb{N}}$ of clopen sets in $X$ with $\bigcap_{n\in\mathbb{N}}C_n=A$ and $\overline{D}(A)=\lim_{n\to\infty}\overline{D}(C_n)$.
	
\subsection{Comparison and almost finiteness (in measure)}
We now recall some of the theory of almost finiteness (in measure), as developed by Kerr \cite{Ker_dimension_2020} and Kerr and Szab{\'o} \cite{KerSza_almost_2020}.
We begin with dynamical subequivalence and comparison.

\begin{df}\label{df:comparison}\cite[Definition~3.1 and Definition 3.2]{Ker_dimension_2020}
Let $G\curvearrowright X$ be an action of a discrete group on a compact Hausdorff space. Given subsets $A,B\subseteq X$, we say that $A$ is \emph{dynamically subequivalent} to $B$, and write $A\precsim B$, if for every closed subset $C\subseteq A$ there are a finite open cover $\mathcal{U}$ of $C$ and group elements $g_U\in G$ for $U\in\mathcal{U}$, such that the sets $g_U\cdot U$, for $U\in\mathcal{U}$, are pairwise disjoint and contained in $B$. We say that $G\curvearrowright X$ has \emph{dynamical comparison} if for all subsets $A,B\subseteq X$, the condition $\mu(A)<\mu(B)$ for all $\mu\in M_G(X)$ implies $A\precsim B$.

\end{df}

Next we recall the notion of a castle.

\begin{df}\label{df:castle}\cite[Definition~5.7]{Ker_dimension_2020}
Let $G\curvearrowright X$ be an action of a discrete group on a compact Hausdorff space. A \emph{tower} for the action is a pair $(S,V)$, where $S\subseteq G$ is a finite subset, and $V\subseteq X$ is a subset of $X$, such that the sets $sV$, for $s\in S$ (called the \emph{levels} of the tower), are pairwise disjoint. The set $S$ is called the \emph{shape} of the tower, and we refer to $V$ as the \emph{base} of the tower. A \textit{castle} is a finite collection $\mathcal{C}=\{(S_i,V_i)\}_{i\in I}$ of towers, so that all levels of all towers are pairwise disjoint. The disjoint union $\bigsqcup_{i\in I}S_i V_i$ is called the \emph{footprint} of the castle.

We say that a castle is \emph{open} (respectively, closed) when  all bases of all towers are open (respectively, closed).
\end{df}

The following is \cite[Definition~8.2]{Ker_dimension_2020} and \cite[Definition 3.5]{KerSza_almost_2020}. We assume minimality for convenience when stating the definition of almost finiteness, since this is the only case we will be interested in.

\begin{df}\label{def almost finite in measure}
A minimal action $G\curvearrowright X$ of a discrete group on a compact metric space is said to be \emph{almost finite in measure} if for any finite subset $K\subseteq G$ and any $\delta,\varepsilon>0$, there exists an open castle $\mathcal{C}=\{(S_i,V_i)\}_{i\in I}$ such that

\begin{enumerate}[label=(\roman*)]

\item each level $sV_i$, for $s\in S_i$ and $i\in I$, has diameter at most $\delta$,

\item each shape $S_i$ is $(K,\delta)$-invariant,

\item the footprint of $\mathcal{C}$ has lower Banach density at least $1-\varepsilon$, that is,
\[
\underline{D}\Big(\bigsqcup_{i\in I}S_i\cdot V_i\Big)\ge1-\varepsilon.
\]
\end{enumerate}

We say that $G\curvearrowright X$ is \emph{almost finite} if for any finite subset $K\subseteq G$, any non-empty open subset $B\subseteq X$ and any $\delta>0$, there exists an open castle $\mathcal{C}=\{(S_i,V_i)\}_{i\in I}$ satisfying conditions (i) and (ii) above in addition to
\begin{enumerate}

\item[(iii')] the complement of the footprint of $\mathcal{C}$ is dynamically subequivalent to $B$, that is,
\[
X\setminus \bigsqcup_{i\in I}S_i\cdot V_i\precsim B.
\]

\end{enumerate}
\end{df}

The following has been mentioned in \cite{Jos23}.

\begin{lma}\label{rem:EssFreeAlmostFinite}
Let $G$ be an amenable, countable group acting on a compact metrizable space $X$. If $G\curvearrowright X$ is almost finite in measure, then it is essentially free.
\end{lma}

\begin{proof}
Let $g\in G\setminus\{1\}$ and $\mu\in M_G(X)$ be given. Fix $\varepsilon>0$, and let $\varepsilon'>0$ satisfy $1-(1-\varepsilon')^2=\varepsilon$. Set $\mathrm{Fix}(g)=\{x\in X\colon g\cdot x=x\}$. We will show that $\mu(\mathrm{Fix}(g))<\varepsilon$. Find an open castle $\{(V_i,S_i)\}_{i\in I}$ with shapes that are $(\{g^{-1}\},\varepsilon')$-invariant, and satisfying 
\begin{equation}\label{2.4}\tag{2.4}
\underline{D}\Big(\bigsqcup_{i\in I}S_i\cdot V_i\Big)\ge1-\varepsilon'.
\end{equation}
Given $i\in I$, it is easy to see that $(S_i\cap g^{-1}S_i)\cdot V_i$ is disjoint from $\mathrm{Fix}(g)$. Thus $\bigsqcup_{i\in I}(S_i\cap g^{-1}S_i)\cdot V_i\subseteq X\setminus\mathrm{Fix}(g)$, and we have
\begin{align*}
1-\mu(\mathrm{Fix}(g))&\ge\sum_{i\in I}|S_i\cap g^{-1}S_i|\mu(V_i)\\
&\ge(1-\varepsilon')\sum_{i\in I}|S_i|\mu(V_i)\\
&=(1-\varepsilon')\mu\big(\bigsqcup_{i\in I}S_i\cdot V_i\big)\\
&\stackrel{\text{\eqref{2.4} \eqref{eq densities suprema infima}}}{\ge}{(1-\varepsilon')^2}=1-\varepsilon.
\end{align*}
As $\varepsilon$ is arbitrary, this shows that $\mu(\mathrm{Fix}(g))=0$. Thus $G\curvearrowright X$ is essentially free.
\end{proof}

We end this section by recalling the definition of the small boundary property.

\begin{df}\label{def sbp}
Let $G$ be a discrete group and let $X$ be a compact Hausdorff space. An action $G\curvearrowright X$ is said to have the \textit{small boundary property} if $X$ admits a basis for its topology consisting of sets $U\subseteq X$ such that $\mu(\partial U)=0$ for all $\mu\in M_G(X)$.
\end{df}

\section{Essentially free actions and almost finiteness in measure}\label{sec ess free}

Our main goal in this section is to show that essentially free actions on zero-dimensional spaces are almost finite in measure (\autoref{thm ess free actions on zero dim spaces are afinmeas}). It will then follow from \autoref{rem:EssFreeAlmostFinite} that, for actions on zero-dimensional spaces, essential freeness is in fact \emph{equivalent} to almost finiteness in measure.

Our approach follows \cite[Section~3]{KerSza_almost_2020}, and the lemmas we present below are careful extensions of theirs to the situation where a set of fixed points has null density but may be nonempty. In order to efficiently deal with such subsets, we introduce the following notation.

\begin{nota}
Let $G\curvearrowright X$ be an action of a discrete group on a topological space. For $g\in G$, we set $\mathrm{Fix}(g)= \{x\in X\colon g\cdot x=x\}$. For a finite subset $F\subseteq G$, we set 
\[
X^F_\mathrm{nf}\coloneqq  \bigcup_{g\in F^{-1}F\setminus\{1\}}\mathrm{Fix}(g), \ \ \  \text{ and } \ \ \  X^F_\mathrm{free}\coloneqq X\setminus X^F_\mathrm{nf}.
\] 
Note that $X^F_\mathrm{free}$ is an open set, and moreover we have 
\[
X^F_\mathrm{free}=\big\{x\in X\colon \mbox{the map } F\to X \mbox{ given by } g\mapsto gx \mbox{ is injective}\big\}.
\]
In other words, the restricted action $F\curvearrowright X_{\mathrm{free}}^F$ is free, and $X^F_{\mathrm{nf}}$ is the part of $X$ where $F$ acts non-freely.
\end{nota}

The following is \cite[Definition~3.7]{KerSza_almost_2020}.

\begin{df}
Let $\varepsilon>0$ and let $X$ be a topological space. A collection $\{A_i\}_{i\in I}$ of finite subsets of $X$ is said to be \emph{$\varepsilon$-disjoint} if for each $i\in I$ there exists a subset $A_i'\subseteq A_i$ such that $|A_i'|\ge(1-\varepsilon)|A_i|$ and the sets $A_i'$, for $i\in I$, are pairwise disjoint.   
\end{df}

\begin{lma}\label{lem first in essfree are afinmeas}
Let $G$ be a countably infinite discrete amenable group acting on a compact metrizable space $X$. Assume that the action $G\curvearrowright X$ is essentially free. Let $T, L\subseteq G$ be finite subsets, let $W\subseteq X$ be any subset, let $\beta>0$, and let $\varepsilon\in(0,1)$. Assume that for each $w\in W$ there exists a subset $E_w\subseteq L$ such that the following conditions are satisfied:
\begin{enumerate}[label=(\roman*)]
\item $E_w$ is $\big(T^{-1},\beta(1-\varepsilon)\big)$-invariant for all $w\in W$.
			
\item The map $E_w\to E_w\cdot w$, given by $g\mapsto g\cdot w$,  is bijective for all $w\in W$. In other words, $w\in X^{E_w}_\mathrm{free}$ for all $w\in W$.

\item The collection $\{E_w\cdot w\}_{w\in W}$ is $\varepsilon$-disjoint.

\item The set $A\coloneqq \bigcup_{w\in W}E_w\cdot w$ has positive lower Banach density.
\end{enumerate}
Then, for any F\o lner sequence $(F_n)_{n\in\mathbb{N}}$ for $G$, there exists $n_0\in\mathbb{N}$ such that, for any $n\geq n_0$, there is a closed set $Y_n\subseteq X$ of zero upper Banach density satisfying
\[
\big|\{g\in F_n\colon g\cdot x\in (T^{-1}\cdot A)\triangle A\}\big|<\beta \big|\{g\in F_n\colon g\cdot x\in A\}\big|,
\]
for all $x\in X\setminus Y_n$.
\end{lma}

\begin{proof}
Since the set $L$ is finite, we can find $0<\beta_0<\beta$ such that whenever $L'\subseteq L$ is $(T^{-1},\beta(1-\varepsilon))$-invariant, then $L'$ is actually $(T^{-1},\beta_0(1-\varepsilon))$-invariant. In particular, by (i), for every $w\in W$, the set $E_w$ is $(T^{-1},\beta_0(1-\varepsilon))$-invariant. Set 
\[
K= L L^{-1}(\{1_G\}\cup T),
\] 
and note that $K$ is a finite subset of $G$ containing $LL^{-1}$. 
Let $(F_n)_{n\in\mathbb{N}}$ be a F\o lner sequence for $G$. Employing \eqref{eq densities Folner sets}, we can find $n_0\in\N$ such that for $n\geq n_0$ we have
\begin{enumerate}
\item[(a)] $\underline{D}_{F_n}(A)>\frac{1}{2}\underline{D}(A)$ 
\item[(b)] $F_n$ is $\left(K,\frac{\beta-\beta_0}{2|K|}\underline{D}(A)\right)$-invariant.
\end{enumerate}

For $n\geq n_0$, set $Y_n= X^{F_n}_\mathrm{nf}\subseteq X$, which is a closed subset of $X$. Moreover, it follows from the second part of \eqref{eq densities suprema infima} that the upper Banach density of $Y_n$ is zero, since the action is essentially free. Now, let $x\in X\setminus Y_n$ be given. Since $x$ does not belong to $Y_n$, the map $F_n\to F_n\cdot x$ given by $g\mapsto g\cdot x$ is a bijection. Set 
\[
C=\{z\in X \ \mbox{ such that } \  z\in F_n\cdot x \ \mbox{  and  } \ K\cdot z\not\subseteq F_n\cdot x\}\subseteq X,
\] 
and note that $C\subseteq K^{-1}(KF_n\cdot x\triangle F_n\cdot x)$. Hence, using the definition of $\underline{D}_{F_n}$ at the third step and using at the last step that $x\in X^{F_n}_{\mathrm{free}}$, we have
\begin{align*}\label{eq size of C}\tag{3.1}
\abs{C}&\le\abs{K}\cdot\abs{KF_n\triangle F_n}\\
&\stackrel{\mathrm{(b)}}{\leq}\frac{\beta-\beta_0}{2}\underline{D}(A)\cdot \left|F_n\right|\\
&\leq(\beta-\beta_0)\cdot\frac{1}{2}\underline{D}(A)\cdot\underline{D}_{F_n}(A)^{-1}\cdot\left|\{t\in F_n\colon t\cdot x\in A\}\right|\\
&\stackrel{\mathrm{(a)}}{<}(\beta-\beta_0)\cdot\left|A\cap F_n\cdot x\right|,
\end{align*}
Using $\varepsilon$-disjointness of the sets $\{E_w\cdot w\}_{w\in W}$ and assumption (ii) of the lemma, for every $w\in W$ we can find a subset $E_w'\subseteq E_w$ with $|E'_w|\ge(1-\varepsilon)|E_w|$, such that the sets $E'_w\cdot w$, for $w\in W$, are pairwise disjoint. Set 
\[
W'=\{w\in W\colon E_w\cdot w\subseteq F_n\cdot x\}.
\] 
Then, using assumption (ii) of the lemma at the second step, we get
\begin{equation}\label{eq second estimate}\tag{3.2}
(1-\varepsilon)\sum_{w\in W'}|E_w|\le\sum_{w\in W'}|E'_w|=\sum_{w\in W'}|E'_w\cdot w|=\Big|\bigsqcup_{w\in W'}E'_w\cdot w\Big|\le \big|A\cap F_n\cdot x\big|.
\end{equation}
We claim that 
\begin{equation}\label{eq  claim inclusion}\tag{3.3}
(T^{-1}\cdot A\triangle A)\cap F_n\cdot x\subseteq C\cup\bigcup_{w\in W'}(T^{-1}E_w\cdot w\triangle E_w\cdot w).
\end{equation}
Indeed, first note that
\[
T^{-1}\cdot A\triangle A=\Big(\bigcup_{w\in W}T^{-1}E_w\cdot w\Big)\triangle\Big(\bigcup_{w\in W}E_w\cdot w\Big)\subseteq \bigcup_{w\in W}\big(T^{-1}E_w\cdot w\triangle E_w\cdot w\big).
\]
Let $z\in(T^{-1}\cdot A\triangle A)\cap F_n\cdot x$. If $z\in \bigcup_{w\in W'}(T^{-1}E_w\cdot w\triangle E_w\cdot w)$, then there is nothing to show. Otherwise, there exists $w\in W\setminus W'$ such that $z\in T^{-1}E_w\cdot w\triangle E_w\cdot w$, and, in particular, $z$ belongs to $(\{1_G\}\cup T^{-1})E_w\cdot w$. In this case, we can find $h\in\{1_G\}\cup T^{-1}$ and $g\in E_w$ so that $z=hgw$. By definition of $W'$, we have $E_w\cdot w\not\subseteq F_n\cdot x$, so there exists $g_0\in E_w$ such that $g_0\cdot w\not\in F_n\cdot x$.  Set $s= g_0g^{-1}h^{-1}\in K$. Then 
\[
s\cdot z=g_0\cdot w\not\in F_n\cdot x,
\]
which is to say that $K\cdot z\not\subseteq F_n\cdot x$. Thus, we have $z\in C$, which finishes the proof of the claim (that is, the inclusion in \eqref{eq  claim inclusion}).

To conclude, using that $x\in X_{\mathrm{free}}^{F_n}$ at the first step, and using $\big(T^{-1},\beta_0(1-\varepsilon)\big)$-invariance of $E_w$ at the fourth step, we have
\begin{align*}
\left|\{g\in F_n\colon  g\cdot x\in (T^{-1}\cdot A)\triangle A\}\right| &= \left|(T^{-1}\cdot A\triangle A)\cap F_n\cdot x\right|\\
&\stackrel{\mathmakebox[\widthof{<}]{\text{\eqref{eq  claim inclusion}}}}{\le} \left|C\right|+\sum_{w\in W'}\left|T^{-1}E_w\cdot w\triangle E_w\cdot w\right|\\
&\le \left|C\right|+\sum_{w\in W'}\left|T^{-1}E_w\triangle E_w\right|\\
&\stackrel{\mathmakebox[\widthof{<}]{\text{\eqref{eq size of C}}}}{<} (\beta-\beta_0)\left|A\cap F_n\cdot x\right|+\beta_0(1-\varepsilon)\sum_{w\in W'}\left|E_w\right|\\
&\stackrel{\mathmakebox[\widthof{<}]{\text{\eqref{eq second estimate}}}}{\le} (\beta-\beta_0)\left|A\cap F_n\cdot x\right|+\beta_0\left|A\cap F_n\cdot x\right| \\
& =\beta\left|A\cap F_n\cdot x\right| =\beta\left|\{g\in F_n\colon g\cdot x\in A\}\right|.
\end{align*}

Since this is true for any $x\in X\setminus Y_n$, the proof is complete.
\end{proof}

\begin{lma}\label{lem second in essfree are afinmeas}
Let $G$ be a countably infinite discrete amenable group acting on a compact metrizable zero-dimensional space $X$. Let $T\subseteq G$ be a finite set containing the unit of $G$, and let $\varepsilon,\beta>0$ satisfy $\varepsilon(1+\beta)<1$. Let $A\subseteq B\subseteq X$ be clopen sets such that:
\begin{enumerate}[label=(\roman*)]
\item $A$ satisfies the conclusion of \autoref{lem first in essfree are afinmeas}, that is: for any F\o lner sequence $(F_n)_{n\in\mathbb{N}}$ for $G$, there exists $n_0\in\mathbb{N}$ such that, for any $n\geq n_0$, there is a closed set $Y_n\subseteq X$ of zero upper Banach density satisfying
\[
\big|\{g\in F_n\colon g\cdot x\in (T^{-1}\cdot A)\triangle A\}\big|<\beta \big|\{g\in F_n\colon g\cdot x\in A\}\big|,
\]
for all $x\in X\setminus Y_n$.

\item $\underline{D}_T(B)\ge\varepsilon$.

\end{enumerate}
Then
\[
\underline{D}(B)\ge(1-\varepsilon(1+\beta))\underline{D}(A)+\varepsilon.
\]
\end{lma}

\begin{proof}
The conclusion follows immediately from (ii) if $\underline{D}(A)=0$, so assume that this is not the case. Let $\theta>0$ satisfy $\theta< \underline{D}(A)$. By considering a F\o lner sequence for $G$, we obtain a finite subset $F\subseteq G$ such that

\begin{enumerate}

\item[(a)] $F$ is $(T,\theta)$-invariant.
			
\item[(b)] $\underline{D}_F(A)>\underline{D}(A)-\theta$.

\item[(c)] there is a closed set $Y\subseteq X$ with $\overline{D}(Y)=0$ such that for all $x\in X\setminus Y$ we have
\[
\left|\{g\in F\colon g\cdot x\in T^{-1}\cdot A\triangle A\}\right|<\beta \left|\{g\in F\colon g\cdot x\in A\}\right|.
\]

\end{enumerate}
Condition~(c) implies that for all $x\in X\setminus Y$, we have 
\begin{equation}\label{eq 2nd lem 1 plus beta}\tag{3.4}
\left|\{g\in F\colon g\cdot x\in T^{-1}\cdot A\}\right|<(1+\beta)\left|\{g\in F\colon g\cdot x\in A\}\right|.
\end{equation}
By $(T,\theta)$-invariance of $F$, we get 
\begin{align*}\label{eq 2nd lemma estimate TF}\tag{3.5}
\left|TF\right|<(1+\theta)\left|F\right|.
\end{align*}
For every $x\in X\setminus Y$, set $\alpha_x\coloneqq \frac{1}{\left|F\right|}\left|\{g\in F\colon g\cdot x\in A\}\right|$ and note that 
\begin{equation}\label{eqalphax}\tag{3.6}
\alpha_x\ge\underline{D}_F(A)>\underline{D}(A)-\theta.
\end{equation}
Define $F'=\{g\in F\colon A\cap (Tg\cdot x)=\emptyset\}$, and observe that 
\[
F\setminus F'=\{g\in F\colon g\cdot x\in T^{-1}\cdot A\}.
\] 
Thus,
\begin{align*}\label{eq 2nd lemma estimate alphax}\tag{3.7}
\frac{\left|F'\right|}{\left|F\right|}&=1-\frac{\left|F\setminus F'\right|}{\left|F\right|}\\
&=1-\frac{\left|\{g\in F\colon g\cdot x\in T^{-1}\cdot A\}\right|}{\left|\{g\in F\colon g\cdot x\in A\}\right|}\cdot\frac{\left|\{g\in F\colon g\cdot x\in A\}\right|}{\left|F\right|}\\
&\stackrel{\mathmakebox[\widthof{=}]{\text{\eqref{eq 2nd lem 1 plus beta}}}}{>}1-(1+\beta){\alpha_x}.
\end{align*}
Fix $g\in G$. Since $\underline{D}_T(B)\ge\varepsilon$ by assumption, we have $\left|\{t\in T\colon tg\cdot x\in B\}\right|\ge\varepsilon\left|T\right|$. If $g\in F'$, then we also have $A\cap Tg\cdot x=\emptyset$. Combining these facts, we get $\left|\{t\in T\colon tg\cdot x\in B\setminus A\}\right|\ge\varepsilon\left|T\right|$, and hence 
\[
\left|\{(g,t)\in F'\times T\colon tg\cdot x\in B\setminus A\}\right|\ge\varepsilon\cdot\left|F'\right|\cdot \left|T\right|.
\] 
Thus, there exists $t^*\in T$ such that $\left|\{g\in F'\colon t^*g\cdot x\in B\setminus A\}\right|\ge\varepsilon\left|F'\right|$. In the following calculation, we use at the second step that $T$ contains the identity of $G$, to get:
\begin{align*}\label{long eq}\tag{3.8}
&\frac{\left|\{s\in TF\colon  s\cdot x\in B\}\right|}{\left|TF\right|}\\
&=\frac{\left|\{s\in TF\colon s\cdot x\in A\}\right|}{\left|TF\right|}+\frac{\left|\{s\in TF\colon s\cdot x\in B\setminus A\}\right|}{\left|TF\right|}\\
&\ge \frac{\left|\{g\in F\colon g\cdot x\in A\}\right|}{\left|F\right|}\cdot\frac{\left|F\right|}{\left|TF\right|}+\frac{\left|\{s\in t^*F\colon s\cdot x\in B\setminus A\}\right|}{\left|F'\right|}\cdot\frac{\left|F'\right|}{\left|F\right|}\cdot\frac{\left|F\right|}{\left|TF\right|} \\
&\stackrel{\mathmakebox[\widthof{=}]{\text{\eqref{eq 2nd lemma estimate TF}}}}{>} \frac{1}{1+\theta} \Big(\frac{\left|\{g\in F\colon g\cdot x\in A\}\right|}{\left|F\right|}+\frac{\left|\{s\in t^*F\colon s\cdot x\in B\setminus A\}\right|}{\left|F'\right|}\cdot\frac{\left|F'\right|}{\left|F\right|}\Big) \\
&\stackrel{\mathmakebox[\widthof{=}]{\text{\eqref{eq 2nd lemma estimate alphax}}}}{>} \frac{\alpha_x}{1+\theta}+\frac{\varepsilon(1-(1+\beta)\alpha_x)}{1+\theta}\\
&= \frac{(1-\varepsilon(1+\beta))\alpha_x+\varepsilon}{1+\theta}\\
&\stackrel{\mathmakebox[\widthof{=}]{\text{\eqref{eqalphax}}}}{\ge} \frac{(1-\varepsilon(1+\beta))(\underline{D}(A)-\theta)+\varepsilon}{1+\theta}.
\end{align*}
By \eqref{eq upper dens lim of nbds}, by zero dimensionality of $X$, and by condition~(c) above, there is a clopen set $C\subseteq X$ with $Y\subseteq C$ and $\overline{D}(C)<\frac{\theta}{\left|TF\right|}$. Using at the third step that $Y\subseteq C$, we get
\begin{align*}\label{eq_BTFC}\tag{3.9}
\underline{D}\big(B\cup TFC\big)&\ge\underline{D}_{TF}\big(B\cup TFC\big)\\
&=\inf_{x\in X}\frac{\left|\{s\in TF\colon s\cdot x\in B\cup TFC\}\right|}{\left|TF\right|}\\
&=\inf_{x\in X\setminus Y}\frac{\left|\{s\in TF\colon s\cdot x\in B\cup TFC\}\right|}{\left|TF\right|}\\
&\ge\inf_{x\in X\setminus Y}\frac{\left|\{s\in TF\colon s\cdot x\in B\}\right|}{\left|TF\right|}\\
&\stackrel{\mathmakebox[\widthof{=}]{\eqref{long eq}}}{\ge}\frac{(1-\varepsilon(1+\beta))(\underline{D}(A)-\theta)+\varepsilon}{1+\theta}.
\end{align*}
Let $\mu\in M_G(X)$. Then
\begin{align*}\label{eqmuB}\tag{3.10}
\mu(B)&\ge\mu(B\cup TFC)-\mu(TFC)\\
&\stackrel{\mathmakebox[\widthof{=}]{\eqref{eq densities suprema infima}}}{\geq}\underline{D}(B\cup TFC)-\left|TF\right|\cdot\overline{D}(C)\\
&\ge\underline{D}(B\cup TFC)-\theta.
\end{align*}
Using \eqref{eq_BTFC} and \eqref{eqmuB} at the second step, we deduce that
\[
\underline{D}(B)=\inf_{\mu\in M_G(X)}\mu(B)   \stackrel{}{\ge}\frac{(1-\varepsilon(1+\beta))(\underline{D}(A)-\theta)+\varepsilon}{1+\theta}-\theta.
\]
Taking limits as $\theta\to 0^+$, we conclude that $\underline{D}(B)\ge(1-\varepsilon(1+\beta))\underline{D}(A)+\varepsilon$, as desired.
\end{proof}

\begin{lma}\label{lem 3rd in afinmeas}
Let $G$ be a countably infinite, discrete, amenable group acting on a compact, metrizable, zero-dimensional space $X$. Let $S\subseteq G$ be a finite subset, let $\delta>0$ and $\varepsilon \in \big(0,\tfrac{1}{2}\big)$ be given, and let $Y,Z\subseteq X$ be clopen sets with $X^S_\mathrm{nf}\subseteq Z$. Then there is a clopen castle $\mathcal{C}=\{(S_i,V_i)\}_{i\in I}$ with the following properties:
\begin{enumerate}[label=(\roman*)]
\item The bases of $\mathcal{C}$ are pairwise disjoint subsets of $X\setminus Z$.

\item Each level of $\mathcal{C}$ has diameter less than $\delta$.

\item We have $S_i\subseteq S$ and $\left|S_i\right|\ge(1-\varepsilon)\left|S\right|$, for all $i\in I$.

\item The footprint $C=\bigsqcup_{i\in I}S_i\cdot V_i$ of $\mathcal{C}$ satisfies:
\begin{enumerate}[label=(\arabic*)]
\item[(a)] $Y\cap C=\emptyset$.
\item[(b)] $Y\sqcup C= Y\cup\bigcup_{i\in I}S\cdot V_i$.
\item[(c)] $\left|(Y\sqcup C)\cap S\cdot x\right|\ge\varepsilon\left|S\right|$ for all $x\in X\setminus Z$.
\end{enumerate}
\end{enumerate}
\end{lma}

\begin{proof}
Since $X\setminus Z\subseteq X^S_\mathrm{free}$, using compactness and zero dimensionality of $X\setminus Z$, we can produce a clopen partition $\{V_j\}_{j=1}^m$ of $X\setminus Z$ so that, for every $j=1,\ldots,m$, the pair $(S,V_j)$ is a tower with levels of diameter less than $\delta$.\footnote{For example, start with a clopen partition of $X\setminus Z$ whose elements have diameter less than $\delta$. The images of these sets under $S$ may in general have larger diameter, so one subdivides the partition further to guarantee that all the translates also have small diameter.}

Set 
\[
\mathcal{J}=\big\{T\subseteq S\colon \left|T\right|\ge(1-\varepsilon)\left|S\right|\big\},
\]
and write $A_0\coloneqq Y$. For $T\in\mathcal{J}$, we define
\begin{equation*}
V_{1,T}\coloneqq V_1\cap\bigg(\bigcap_{s\in S\setminus T}s^{-1}\cdot A_0\bigg)\cap\bigg(\bigcap_{s\in T}s^{-1}\cdot (X\setminus A_0)\bigg)\subseteq X\setminus Z.
\end{equation*}

Note that the collection $\big\{(T, V_{1,T})\big\}_{T\in\mathcal{J}}$ is a castle whose footprint does not intersect $A_0$. Set $C_1=\bigsqcup_{T\in\mathcal{J}}T\cdot V_{1,T}$, and define $A_1= A_0\sqcup C_1$. For $T\in\mathcal{J}$, we further set
\begin{equation*}
V_{2,T}\coloneqq V_2\cap\bigg(\bigcap_{s\in S\setminus T}s^{-1}\cdot A_1\bigg)\cap\bigg(\bigcap_{s\in T}s^{-1}\cdot (X\setminus A_1)\bigg)\subseteq X\setminus Z.
\end{equation*}

One verifies that $\big\{(T,V_{2,T})\big\}_{T\in\mathcal{J}}$ is a castle whose footprint does not intersect $A_1$. Next, set $C_2=\bigsqcup_{T\in\mathcal{J}}T\cdot V_{2,T}$ and $A_2= A_1\sqcup C_2$. Proceeding inductively, we obtain castles $\mathcal{C}_i=\{(T,V_{i,T})\}_{T\in\mathcal{J}}$, for $i=1,\ldots,m$, with pairwise disjoint footprints, all of which are disjoint from $Y$, and with bases given by
\begin{align*}\label{df ViT}\tag{3.11}
V_{i,T}=V_i\cap\bigg(\bigcap_{s\in S\setminus T}s^{-1}\cdot A_{i-1}\bigg)\cap\bigg(\bigcap_{s\in T}s^{-1}\cdot (X\setminus A_{i-1})\bigg)\subseteq X\setminus Z.
\end{align*}
We will also need the nested sets 
\[
Y=A_0\subseteq  A_1 \subseteq \cdots \subseteq A_m\subseteq X,
\]
defined as $A_j=Y\sqcup\bigsqcup_{\substack{T\in\mathcal{J}\\ 1\le i\le j}}T\cdot V_{i,T}$ for $j=1,\ldots,m$.

Denote by $\mathcal{C}$ the union of $\mathcal{C}_1,\ldots,\mathcal{C}_m$, that is, $\mathcal{C}=\{(T,V_{i,T})\}_{\substack{T\in\mathcal{J}\\ 1\le i\le m}}$. We claim that $\mathcal{C}$ satisfies the properties listed in the statement of the lemma. 

Properties (i), (ii) and (iii) are evident, so we turn to (iv). Denoting by $C$ the footprint of $\mathcal{C}$, it follows that $Y\cap C=\emptyset$ because the footprints of the $\mathcal{C}_i$ are disjoint from $Y$. This shows (a).
		To verify condition~(b), note first that the fact that $S_i\subseteq S$ for all $i=1,\ldots,m$ implies that
\[
Y\sqcup C\subseteq Y\cup\bigcup_{\substack{T\in\mathcal{J}\\ 1\le i\le m}}S\cdot V_{i,T}.
\]
For the reverse inclusion, it suffices to show that if $x\notin Y$ and $x\in \bigcup_{\substack{T\in\mathcal{J}\\ 1\le i\le m}}S\cdot V_{i,T}$, then $x\in C$. Let $x$ satisfy these conditions, and set

\begin{equation*}
i_0\coloneqq\min\Big\{i=1,\dots,m\colon x\in\bigcup_{T\in\mathcal{J}}S\cdot V_{i,T}\Big\}.
\end{equation*}
Then there exist $s\in S$, $T\in\mathcal{J}$ and $x'\in V_{i_0,T}$ such that $x=s\cdot x'$. If $s\in T$, then $x\in C$ and we are done.

If $s\not\in T$, then by \eqref{df ViT} we have $x'\in s^{-1}\cdot A_{i_0-1}$, and thus $x=s\cdot x'\in A_{i_0-1}$. Since $x\not\in Y$, there is $1\le j<i_0$ so that \[
x\in\bigsqcup_{T\in\mathcal{J}}T\cdot V_{j,T}\subseteq\bigcup_{T\in\mathcal{J}}S\cdot V_{j,T}.
\] 
This contradicts the definition of $i_0$, and shows that $s$ must belong to $T$. We have established (b).
		
Finally, we prove condition~(c) in property~(iv). Let $x\in X\setminus Z$ and find $i=1,\ldots,m$ such that $x\in V_i$. We first treat the case where there exists $T\in\mathcal{J}$ such that $x\in V_{i,T}$. Then $T\cdot x\subseteq T\cdot V_{i,T}\subseteq C$. Using $x\in X^S_\mathrm{free}$ at the third step, we get
\begin{equation*}
\left|(Y\sqcup C)\cap S\cdot x\right|\ge\left|(Y\sqcup C)\cap T\cdot x\right|=\left|T\cdot x\right|=\left|T\right|\ge(1-\varepsilon)\left|S\right|\ge\varepsilon\left|S\right|,
\end{equation*}
as desired. Assume now that $x\not\in\bigcup_{T\in\mathcal{J}}V_{i,T}$, we claim that 
\begin{align*}\label{ineq}\tag{3.12}
\left|\{s\in S\colon s\cdot x\in A_{i-1}\}\right|\ge\varepsilon\left|S\right|.
\end{align*}
Arguing by contradiction, assume that the inequality fails. Define $T_x= \{s\in S\colon s\cdot x\not\in A_{i-1}\}$. Then $T_x$ belongs to $\mathcal{J}$, and it is clear that $x\in V_{i,T_x}$. This contradicts our assumption that $x\not\in\bigcup_{T\in\mathcal{J}}V_{i,T}$, and proves inequality \eqref{ineq}. Using at the second step again that $x\in X^S_\mathrm{free}$, we have
\begin{equation*}
\left|(Y\sqcup C)\cap S\cdot x\right|\ge\left|A_{i-1}\cap S\cdot x\right|=\left|\{s\in S\colon s\cdot x\in A_{i-1}\}\right| \stackrel{\eqref{ineq}}{\ge}\varepsilon\left|S\right|.
\end{equation*}
This completes the proof of (c) and of the lemma.
\end{proof}

The following is the main result of this section.

\begin{thm}\label{thm ess free actions on zero dim spaces are afinmeas}
Let $G$ be a countably infinite, discrete, amenable group acting on a compact metrizable, zero-dimensional space $X$. Then $G\curvearrowright X$ is 
almost finite in measure if and only if it is essentially free.
\end{thm}

\begin{proof}
The ``only if'' implication is the content of \autoref{rem:EssFreeAlmostFinite}, so we will prove the ``if'' part of the statement. Let $K\subseteq G$ be a finite subset, let $\delta>0$, and let $\varepsilon\in(0,\tfrac{1}{2})$ be small enough so that the statement in \autoref{lem invariance} is satisfied. Let $n\in\N$ satisfy $(1-\varepsilon)^n<\varepsilon$. Fix $\beta>0$ and $\alpha\in(0,1)$ so that $\varepsilon(1+\beta)<1$ and
\begin{equation}\label{eq geometric0}\tag{3.13}
\alpha\cdot \frac{1-(1-\varepsilon(1+\beta))^n}{1+\beta}>1-\varepsilon.
\end{equation}

By recursively employing \autoref{rmk Folner}, we can find $(K,\varepsilon)$-invariant finite sets $F_1\subseteq\dots\subseteq F_n\subseteq G$ with $1_G\in F_1$ such that, for each $j=2,\ldots,n$, the set $F_j$ is also $\big(F_i^{-1},\beta(1-\varepsilon)\big)$-invariant, for all $i=1,\ldots,j-1$.

We follow an algorithm to construct, for each $k=1,\ldots, n$, a clopen castle 
$\mathcal{C}_k=\{(S_{i,k},V_{i,k})\}_{i\in I_k}$ such that, if we denote by $C_k$ its footprint, then:
\begin{enumerate}[label=(\arabic*)]
\item All the levels of $\mathcal{C}_k$ have diameter less than $\delta$.
			
\item We have $S_{i,k}\subseteq F_{n-k+1}$ and $\left|S_{i,k}\right|\ge(1-\varepsilon)\cdot\left|F_{n-k+1}\right|$ for all $i\in I_k$.

\item The sets $C_1,\dots, C_k$ are pairwise disjoint.

\item $\bigsqcup\limits_{j=1}^kC_j=\bigsqcup\limits_{j=1}^{k-1}C_j\cup\bigcup\limits_{i\in I_k}F_{n-k+1}\cdot V_{i,k}$.

\item $\underline{D}\Big(\bigsqcup\limits_{j=1}^kC_j\Big)\ge\alpha\varepsilon\sum\limits_{j=0}^{k-1}(1-\varepsilon(1+\beta))^j$.
\end{enumerate}

For the first step, namely for $k=1$, we use essential freeness and \eqref{eq upper dens lim of nbds} to choose a clopen set $Z_1$ containing $X^{F_n}_\mathrm{nf}$ and satisfying 
\begin{align*}\label{small D}\tag{3.14}
\overline{D}(Z_1)<\frac{(1-\alpha)\varepsilon}{\left|F_n\right|}.
\end{align*}
We now apply \autoref{lem 3rd in afinmeas} with $S=F_n$, $Y=\emptyset$ and $Z=Z_1$ to obtain a castle $\mathcal{C}_1=\{(S_{i,1},V_{i,1})\}_{i\in I_1}$ satisfying the conclusion of said lemma. It is evident that conditions (1), (2), (3) and (4) are satisfied, so we check (5). Let $x\in X\setminus Z_1$. Then $\left|C_1\cap F_n\cdot x\right|\ge\varepsilon\left|F_n\right|$. Using this at the last step, we get
\begin{equation*}
\big|\{s\in F_n\colon s\cdot x\in C_1\cup F_n\cdot Z_1\}\big|\ge\big|\{s\in F_n\colon s\cdot x\in C_1\}\big|=\big|C_1\cap F_n\cdot x\big|\ge\varepsilon\big|F_n\big|.
\end{equation*}
On the other hand, for $x\in Z_1$, we (trivially) have
\begin{equation*}
\big|\{s\in F_n\colon s\cdot x\in C_1\cup F_n\cdot Z_1\}\big|=\big|F_n\big|\ge\varepsilon\big|F_n\big|.
\end{equation*}
In particular, we get
\[
\underline{D}(C_1\cup F_n\cdot Z_1)\ge\underline{D}_{F_n}(C_1\cup F_n\cdot Z_1)\ge\varepsilon.
\] 
For $\mu\in M_G(X)$, we have $\mu(C_1\cup F_n\cdot Z_1)\le\mu(C_1)+\left|F_n\right|\cdot\mu(Z_1)$, and, by \eqref{eq densities suprema infima}, we deduce that 
\[
\varepsilon\le\underline{D}(C_1)+\left|F_n\right|\cdot \overline{D}(Z_1)\stackrel{\eqref{small D}}{<}\underline{D}(C_1)+(1-\alpha)\varepsilon.
\] 
We conclude that $\underline{D}(C_1)\ge\alpha\varepsilon$, thus verifying condition~(5) for $k=1$.

Assume that the algorithm has run up to stage $k-1<n$, for some $k\geq 2$. Use essential freeness and \eqref{eq upper dens lim of nbds} to find a clopen set $Z_k$ which contains $X_\mathrm{nf}^{F_{n-k+1}}$ and satisfies $\overline{D}(Z_k)<\frac{(1-\alpha)\varepsilon}{\left|F_{n-k+1}\right|}$. Apply \autoref{lem 3rd in afinmeas} with $S=F_{n-k+1}$, $Y=\bigsqcup_{j=1}^{k-1}C_j$, and $Z=Z_k$, in order to obtain a castle $\mathcal{C}_k=\{(S_{i,k},V_{i,k})\}_{i\in I_k}$. As before, conditions (1), (2), (3) and (4) of the list above follow directly from the statement of \autoref{lem 3rd in afinmeas}, so it suffices to verify condition~(5). As in the first step, we see that 
\[
\underline{D}_{F_{n-k+1}}\bigg(F_{n-k+1}\cdot Z_k\cup \bigsqcup_{j=1}^kC_j\bigg)\ge\varepsilon.
\] 
Therefore, with $T=F_{n-k+1}$, $A=\bigsqcup_{j=1}^{k-1}C_j$ and $B=F_{n-k+1}\cdot Z_k\cup\bigsqcup_{j=1}^kC_j$, condition~(ii) in the statement of \autoref{lem second in essfree are afinmeas} is satisfied. Assume for a moment that $A$ satisfies condition~(i) of \autoref{lem second in essfree are afinmeas} as well; we will prove this shortly. Then, applying \autoref{lem second in essfree are afinmeas} at the first step and using condition~(5) of the algorithm for stage $k-1$ at the second step, we get
\begin{align*}
\underline{D}\bigg(F_{n-k+1}\cdot Z_k \cup\bigsqcup_{j=1}^kC_j\bigg)&\ge(1-\varepsilon(1+\beta))\cdot\underline{D}\bigg(\bigsqcup_{j=1}^{k-1}C_j\bigg)+\varepsilon\\
&\ge\alpha\varepsilon\sum_{j=1}^{k-1}(1-\varepsilon(1+\beta))^j+\varepsilon.
\end{align*}
As argued at the first step of the algorithm, one sees that 
\begin{align*}\label{eq geometric2}
\underline{D}\bigg(F_{n-k+1}\cdot Z_k \cup\bigsqcup_{j=1}^kC_j\bigg)&\le\underline{D}\bigg(\bigsqcup_{j=1}^kC_j\bigg)+\left|F_{n-k+1}\right|\cdot\overline{D}(Z_k)\\
&<\underline{D}\bigg(\bigsqcup_{j=1}^kC_j\bigg)+(1-\alpha)\varepsilon.
\end{align*}
Combining the last two inequalities, we obtain 
\[
\underline{D}\bigg(\bigsqcup_{j=1}^kC_j\bigg)\ge\alpha\varepsilon\sum_{j=0}^{k-1}(1-\varepsilon(1+\beta))^j,
\] 
as required. (Note that the above sum starts at $j=0$, while the
previous one does at $j=1$.)

Thus, the problem of constructing $\mathcal{C}_k$ has been reduced to showing that the set $A=\bigsqcup_{j=1}^{k-1}C_j$ satisfies condition~(i) of \autoref{lem second in essfree are afinmeas}. To that end, set $W=\bigcup_{j=1}^{k-1}\bigsqcup_{i\in I_j}V_{i,j}$ and for each $w\in W$, let $j_w$ denote the smallest $j \in\{1,\dots,k-1\}$ such that $w\in\bigsqcup_{i\in I_j}V_{i,j}$. We wish to apply \autoref{lem first in essfree are afinmeas} with $L=F_n$, $T=F_{n-k+1}$, and $E_w=F_{n-j_w+1}$ for $w\in W$, so we will first check that its assumptions are met. Conditions~(i) and~(ii) of \autoref{lem first in essfree are afinmeas} follow by construction. To see that the collection $\{E_w\cdot w\}_{w\in W}$ is $\varepsilon$-disjoint, given $w\in W$ we let $i$ be the unique index in $I_{j_w}$ such that $w\in V_{i,{j_w}}$, and set $E_w'= S_{i,j_w}$. Note that $\{E_w'\cdot w\}_{w\in W}$ is a collection of pairwise disjoint sets, and that $\left|E_w'\right|\ge(1-\varepsilon)\left|E_w\right|$, since $w\in X^{E_w}_{\mathrm{free}}$. It is left to show that $\bigcup_{w\in W}E_w\cdot w=\bigsqcup_{j=1}^{k-1}C_j$. By definition of the sets $E_w$, we have 
\begin{align*}
\bigcup_{w\in W}E_w\cdot w&=F_n\cdot \bigg(\bigsqcup_{i\in I_1}V_{i,1}\bigg)\cup F_{n-1}\cdot \bigg(\bigsqcup_{i\in I_2}V_{i,2}\setminus\bigsqcup_{i\in I_1}V_{i,1}\bigg)\cup \cdots\\ 
& \ \ \ \cdots \cup F_{n-k+2}\cdot \bigg(\bigsqcup_{i\in I_{k-1}}V_{i,k-1}\setminus\bigcup_{j=1}^{k-2}\bigsqcup_{i\in I_j}V_{i,j}\bigg).
\end{align*}
In the next chain of equalities, we apply condition~(4) of the algorithm recursively, from $k-1$ backwards to $1$, and we also use that $C_1=F_n\cdot \left(\bigsqcup_{i\in I_1} V_{i,1}\right)$, which is true by part~(b) in condition~(iv) of \autoref{lem 3rd in afinmeas}:
\begin{align*}
\bigsqcup_{j=1}^{k-1}C_j&=\bigsqcup_{j=1}^{k-2}C_j\cup F_{n-k+2}\cdot \bigg(\bigsqcup_{i\in I_{k-1}}V_{i,{k-1}}\bigg)\\
&=\ldots\\
&=F_n\cdot \bigg(\bigsqcup_{i\in I_1}V_{i,1}\bigg)\cup F_{n-1}\cdot \bigg(\bigsqcup_{i\in I_2}V_{i,2}\bigg)\cup\dots\cup F_{n-k+2}\cdot \bigg(\bigsqcup_{i\in I_{k-1}}V_{i,k-1}\bigg).
\end{align*}
Using the inclusions $F_1\subseteq\dots\subseteq F_n$, we deduce that 
\begin{align*}
F_n\cdot \bigg(\bigsqcup_{i\in I_1}V_{i,1}\bigg)\cup F_{n-1}\cdot \bigg(\bigsqcup_{i\in I_2}&V_{i,2}\bigg)\cup  \dots\cup F_{n-k+2}\cdot \bigg(\bigsqcup_{i\in I_{k-1}}V_{i,k-1}\bigg)=\\
F_n\cdot \bigg(\bigsqcup_{i\in I_1}V_{i,1}\bigg)\cup F_{n-1}\cdot  &\bigg(\bigsqcup_{i\in I_2} V_{i,2}\setminus\bigsqcup_{i\in I_1}V_{i,1}\bigg) \cup\dots \\
\dots \cup F_{n-k+2}\cdot &\bigg(\bigsqcup_{i\in I_{k-1}}V_{i,k-1}\setminus\bigcup_{j=1}^{k-2}\bigsqcup_{i\in I_j}V_{i,j}\bigg).
\end{align*}

This implies that $\bigcup_{w\in W}E_w\cdot w=\bigsqcup_{j=1}^{k-1}C_j$. Therefore \autoref{lem first in essfree are afinmeas} applies, and both conditions of \autoref{lem second in essfree are afinmeas} are met, as required.
		
After the algorithm has run its full course and we have clopen castles $\mathcal{C}_1,\ldots,\mathcal{C}_n$ as above, we set $\mathcal{C}=\{(S_{i,k},V_{i,k})\}_{\substack{i\in I_k\\ 1\le k\le n}}$, that is, the union of all of them. By condition~(1) of the algorithm, each level in this castle has diameter less than $\delta$. By condition~(2) of the algorithm, as well as \autoref{lem invariance} and our choice of $\varepsilon$, each shape is $(K,\delta)$-invariant. Finally, we have
\begin{align*}
\underline{D}\bigg(\bigsqcup_{\substack{i\in I_k\\ 1\le k\le n}}S_{i,k}\cdot V_{i,k}\bigg)&=\underline{D}\big(\bigsqcup_{j=1}^nC_j\big)\\
&\stackrel{\mathmakebox[\widthof{=}]{\mathrm{(5)}}}{\ge}\alpha\varepsilon \sum_{j=0}^{n-1}(1-\varepsilon(1+\beta))^j\\ 
&=\alpha\varepsilon\frac{1-(1-\varepsilon(1+\beta))^n}{\varepsilon(1+\beta)}\\
&=\alpha\frac{1-(1-\varepsilon(1+\beta))^n}{1+\beta}\\
&\stackrel{\mathmakebox[\widthof{=}]{\text{\eqref{eq geometric0}}}}{>} \ \ 1-\varepsilon.
\end{align*}
This completes the proof.
\end{proof}

We close this section with two interesting consequences of the above theorem. First, we observe that even though the theory developed in \cite{KerSza_almost_2020} (and particularly Sections 4--6 therein) is stated and proved for free actions, once we have \autoref{thm ess free actions on zero dim spaces are afinmeas} at our disposal, some of the results in \cite{KerSza_almost_2020} can be easily extended to cover actions that are essentially free. We state below two such results (corresponding to \cite[Theorem~5.6]{KerSza_almost_2020} and \cite[Theorem~A]{KerSza_almost_2020}) as corollaries. Both proofs are entirely analogous to those in \cite{KerSza_almost_2020}, invoking \autoref{thm ess free actions on zero dim spaces are afinmeas} instead of \cite[Theorem~3.13]{KerSza_almost_2020}. 

\begin{cor}\label{cor ess free sbp iff afinmeas}
Let $G$ be a countably infinite, discrete, amenable group acting on a compact, metrizable space $X$. Assume moreover that the action $G\curvearrowright X$ is essentially free. Then $G\curvearrowright X$ has the small boundary property if and only if it is almost finite in measure.
\end{cor}

The notion of $m$-comparison for $m\in\mathbb{N}$ is a weaker version of comparison, we refer to \cite[Definition~3.2]{Ker_dimension_2020} for the exact definitions.

\begin{cor}\label{cor ess free af iff sbp and comparison}
Let $G$ be a countably infinite, discrete, amenable group acting on a compact, metrizable space $X$. Assume moreover that the action $G\curvearrowright X$ is essentially free. Then, the following statements are equivalent:
\begin{enumerate}
\item The action $G\curvearrowright X$ is almost finite.

\item The action $G\curvearrowright X$ has the small boundary property and comparison.

\item The action $G\curvearrowright X$ has the small boundary property and $m$-comparison, for some $m\ge0$.
\end{enumerate}
\end{cor}
	
\section{Actions on finite-dimensional spaces and groups of polynomial growth}\label{sec:sbp}

In this section, we show that essentially free actions of amenable groups on
\emph{finite-dimensional} spaces have the small boundary property; see \autoref{thm ess free actions on fin dim spaces have sbp}. 
Specializing to groups with polynomial growth, for which comparison has been
shown in \cite{Nar22} to be automatic in the minimal setting, we give necessary and sufficient conditions for classifiability of the associated crossed product; see \autoref{cor polynomial growth}.

We recall a definition due to Kulesza \cite{Kul95} (see also \cite[Section~3]{Lin95}, and recall that $\dim(\emptyset)=-1$).
	
\begin{df}\label{def general position}
Let $X$ be a compact metrizable finite-dimensional space. A family $(A_i)_{i\in I}$ of subsets of $X$ is said to be in \textit{general position} if for any finite subset $F\subseteq I$, we have 
\[
\dim\Big(\bigcap_{i\in F}A_i\Big)\le\max \big\{-1, \dim(X)-\left|F\right|\big\}.
\]
\end{df}

We isolate the following observation for later use.

\begin{rem}\label{rmk general position upper ban dens}
Note that if a collection $(A_i)_{i\in I}$ of subsets of a compact metrizable space $X$ is in general position, then for any finite subset $J\subseteq I$ with $\left|J\right|\geq \dim(X)+1$, we have that $\bigcap_{i\in J}A_i=\emptyset$. In particular, if $G\curvearrowright X$ is an action, $F\subseteq G$ is a finite subset, and $A\subseteq X$ is a subset such that the collection $(g\cdot A)_{g\in F}$ is in general position, then 
\[
\overline{D}(A)\leq \overline{D}^{F^{-1}}(A)\le\frac{\dim(X)}{\left|F\right|}.
\]
\end{rem}

As it turns out, one can obtain families in general position as images under a free action of certain boundaries of open sets:

\begin{lma}\label{lem boundaries in gp}\cite[Lemma~3.1.8]{Sza15diss}
Let $G\curvearrowright X$ be an action of a discrete group on a compact metrizable space. Let $F\subseteq G$ be a finite set, let $K\subseteq  X^F_\mathrm{free}$ be a closed subset of $X$, and let $V$ be an open neighborhood of $K$ such that $K\subseteq V\subseteq X^F_{\mathrm{free}}$. Then, there exists an open set $W\subseteq X$ satisfying $K\subseteq W\subseteq\overline{W}\subseteq V$, such that the family $\big(g\cdot \partial W\big)_{g\in F}$ is in general position.
\end{lma}

We are ready for the main result of this section. 

\begin{thm}\label{thm ess free actions on fin dim spaces have sbp}
Let $G$ be a countably infinite, discrete, amenable group acting on a compact, metrizable, finite-dimensional space $X$. If the action $G\curvearrowright X$ is essentially free, then it has the small boundary property.
\end{thm}

\begin{proof}
Let $\varepsilon\in(0,1)$ and let $U,V\subseteq X$ be open sets satisfying $\overline{U}\subseteq V$. We claim that there exists an open set $U'\subseteq X$ such that 
\[
U\subseteq U'\subseteq\overline{U'}\subseteq V \ \ \mbox{ and } \ \ \overline{D}(\partial U')<\varepsilon.
\] 
Before proceeding, we note that proving the claim is equivalent to establishing the small boundary property, due to \cite[Theorem~5.5]{KerSza_almost_2020} and the fact that the implication (v)$\implies$(i) therein does not require freeness.

Set $d=\dim(X)<\infty$. Since $G$ is infinite, there exists a finite subset $F\subseteq G$ with $\left|F\right|>2d/\varepsilon$. 
		By essential freeness, we have $\overline{D}(X^F_\mathrm{nf})=0$, so by \eqref{eq upper dens lim of nbds} we can find an open set $B\subseteq X$ such that \[
X^F_\mathrm{nf}\subseteq B \ \ \mbox{ and } \ \ \overline{D}(\overline{B})<\frac{\varepsilon}{2}.
\]
Thus the set $\overline{U}\cap(X\setminus B)$ is a closed subset of $V\cap X^F_\mathrm{free}$, which, in turn, is an open subset of $X^F_\mathrm{free}$. Apply \autoref{lem boundaries in gp} to find an open set $W$ with 
\[
\overline{U}\cap(X\setminus B)\subseteq W\subseteq\overline{W}\subseteq V\cap X_\mathrm{free}^F,
\] 
and such that $(g\partial W)_{g\in F}$ is in general position. In particular, by \autoref{rmk general position upper ban dens}, we have $\overline{D}(\partial W)\le d/\left|F\right|<\varepsilon/2$. Since $\overline{U}\subseteq V$, there exists an open set $A\subseteq X$ with $\overline{U}\subseteq A\subseteq\overline{A}\subseteq V$. 

Set $U'= W\cup(A\cap B)$. We want to show that $U\subseteq U'$. For this, let $x\in U$ be given. If $x\in B$, then $x\in A\cap B\subseteq U'$. On the other hand, if $x\not\in B$, then 
\[
x\in U\cap (X\setminus B)\subseteq W\subseteq U',
\]
so in either case we have $x\in U'$, so $U\subseteq U'$. Moreover, we have
\[
\overline{U'}=\overline{W}\cup\overline{A\cap B}\subseteq \overline{W}\cup\overline{A}\subseteq V.
\] 
Finally, 
\[
\partial U'\subseteq\partial W\cup\partial(A\cap B)\subseteq\partial W\cup\overline{B},
\] 
and thus 
\[
\overline{D}(\partial U')\le\overline{D}(\partial W)+\overline{D}(\overline{B})<\frac{\varepsilon}{2}+\frac{\varepsilon}{2}=\varepsilon,
\] 
as desired.
\end{proof}

Specializing to groups of polynomial growth, we show that when the space is finite-dimensional, the necessary conditions of topological freeness and minimality are in fact \emph{sufficient} for  classifiability\footnote{For the sake of this work, we will say that a $C^*$-algebra is \emph{classifiable} if it satisfies the assumptions of the classification theorem stated in the introduction, namely if it is simple, separable, unital, nuclear, $\mathcal{Z}$-stable and satisfies the UCT.} of the crossed product.

\begin{cor}\label{cor polynomial growth}
Let $G$ be a finitely generated group of polynomial growth acting on a compact metrizable, finite-dimensional space $X$. Then $C(X)\rtimes G$ is classifiable if and only if $G\curvearrowright X$ is minimal and topologically free.
\end{cor}
	
\begin{proof}
By \cite{ArcSpi94}, for actions of amenable groups, minimality and topological freeness are equivalent to simplicity of the crossed product. In particular, the ``only if'' implication is automatic.

We turn to the converse, for which, as explained in the introduction, it only remains to check $\mathcal{Z}$-stability of the crossed product. 
 
By Gromov's theorem, finitely generated groups with polynomial growth are virtually nilpotent. In particular, it follows that $G$ has at most countably many subgroups (see, for example, page 2 of the introduction in \cite{Jos21}). By \cite[Corollary~2.4]{Jos21}, any topologically free and minimal action of $G$ is automatically essentially free. Using \autoref{thm ess free actions on fin dim spaces have sbp} we deduce that $G\curvearrowright X$ has the small boundary property.  By the main result of \cite{Nar22}, any minimal action of a finitely generated group with polynomial growth has comparison, and thus \autoref{cor ess free af iff sbp and comparison} implies that $G\curvearrowright X$ is almost finite. Finally, it follows from \cite[Theorem~12.4]{Ker_dimension_2020}, the proof of which does not employ freeness, that $C(X)\rtimes G$ is $\mathcal{Z}$-stable and thus classifiable.
\end{proof}

\section{Joseph's allosteric actions and $\mathcal{Z}$-stability of their crossed products}\label{sec allosteric}
	
At this point we wish to highlight the fact that none of the results in previous sections apply to actions that are topologically free but not essentially free.
	
In the minimal setting, the difference between essential and topological freeness is a subtle one. Following Joseph \cite{Jos21}, a minimal action of an amenable group that is topologically free but not essentially free is called \emph{allosteric} (this agrees with \cite[Definition~1.1]{Jos21}, since $M_G(X)\neq\emptyset$ by amenability of $G$). It was an open problem for some time whether there exists a minimal, allosteric action of an amenable group and the first such examples were recently constructed by Joseph in \cite{Jos23}. It was left as an open question in Joseph's work to determine whether the crossed products of his actions are classifiable.

In this section, we analyze Joseph's work on allosteric actions of amenable groups and show that the crossed products of his examples are $\mathcal{Z}$-stable (and thus classifiable). We do not develop a general framework for establishing $\mathcal{Z}$-stability for crossed products of general topologically free actions, but it is conceivable that the methods developed here are a first step towards extending the theory of almost finiteness to the topologically free setting. We begin with some preliminary lemmas.

\begin{lma}
\label{WeddingCakeLemmaFinite}
Let $\Lambda$ be a group, let $\Lambda_0  \subseteq \Lambda$ be a finite, symmetric subset, let $n\in\mathbb{N}$, let $Y$ be a finite set, let $\Lambda \curvearrowright Y$ be an action, and let $Y_0 \subseteq Y$ be a subset. Then there exists a function $f \colon Y \to [0,1]$ such that
\begin{enumerate}
\item $f(y) = 0$ for all $y \in Y_0$,
			
\item $\abs{f(\lambda\cdot y) - f(y)} \le 1/n$ for all $y \in Y$ and all $\lambda \in \Lambda_0 $, and

\item $\abs{\left\{y \in Y \colon f(y) \neq 1\right\}} \le \big(1 + \abs{\Lambda_0 } + \abs{\Lambda_0 }^2 + \ldots + \abs{\Lambda_0 }^{n-1}\big)\abs{Y_0}$.
\end{enumerate}
\end{lma}

\begin{proof}
Let $\rho \colon Y \times Y \to \N_0 \cup \{\infty\}$ be the distance function in the Shreier graph of $(Y, \Lambda_0 )$, which is given by
\[
\rho(y,z) = \min\big\{m \in \N_0 \colon y = \lambda_1\cdots \lambda_m z,  \mbox{ with } \lambda_1,\ldots,\lambda_m \in \Lambda_0 \big\}     
\]
for all $y,z\in Y$. Given $y\in Y$, we set $\rho(y,Y_0)=\min\limits_{z\in Y_0}\rho(y,z)$. For $y\in Y$, we now define
\[
f(y) = \begin{cases}
1 & \text{if } \rho(y, Y_0) \geq n, \\
\frac{\rho(y,Y_0)}{n} & \text{otherwise}.
\end{cases}
\]
We claim that $f$ satisfies the properties in the statement. Condition (1) is clear. Given $y\in Y$ and $\lambda \in \Lambda_0 $, we have 
\[
\rho(y, Y_0)-1 \leq \rho(\lambda\cdot y,Y_0)\leq \rho(y, Y_0)+1
\] 
and thus $|f(\lambda\cdot y)-f(y)|\leq 1/n$, as desired. Finally, adopting the convention that $\Lambda_0 ^0=\{1\}$, we have
\[
\big\{y\in Y\colon f(y)\neq 1\big\}=\big\{y\in Y\colon \rho(y,Y_0)\leq n-1\big\}=\bigcup_{m=0}^{n-1}\Lambda_0 ^m Y_0,
\]
which has cardinality at most $\big(1 + \abs{\Lambda_0 } + \abs{\Lambda_0 }^2 + \ldots + \abs{\Lambda_0 }^{n-1}\big)\abs{Y_0}$, as desired.
\end{proof}

The following technical lemma is essentially due to Weiss; see \cite[Pages 260-261]{Wei_monotileable_2001}. The result implies that a residually finite, amenable group $\Lambda$ admits a F{\o}lner sequence whose elements are images under sections of finite quotients of $\Lambda$.

\begin{lma}\label{AlmostEquivariantLiftLemma}
Let $\Lambda$ be a countably infinite, residually finite group and let $(\Lambda_n)_{n\in\mathbb{N}}$ be a collection of normal subgroups with finite index such that for every finite subset $F\subseteq \Lambda$ there exists $n\in\mathbb{N}$ such that the restriction of the quotient map $\Lambda\to \Lambda/\Lambda_n$ to $F$ is injective. For $n\in\mathbb{N}$, write $\pi_n\colon \Lambda\to \Lambda/\Lambda_n$ for the quotient map. Then $\Lambda$ is amenable if and only if for every finite subset $K \subseteq \Lambda$ and every $\varepsilon>0$ there exist $n \in \N$ and a map $\varphi \colon \Lambda/\Lambda_n \to \Lambda$ such that:
\begin{enumerate}
\item $\pi_n \circ \varphi = \mathrm{id}_{\Lambda/\Lambda_n}$,

\item $\big|\big\{t \in \Lambda/\Lambda_n \colon \lambda\varphi(t) \neq \varphi(\pi_n(\lambda)t) \ \text{for some} \ \lambda \in K\big\}\big| < \varepsilon \abs{\Lambda/\Lambda_n}$.
\end{enumerate}
In other words, $\varphi$ is an $(K,\varepsilon)$-equivariant section for $\pi_n\colon\Lambda\to \Lambda/\Lambda_n$.
\end{lma}

\begin{proof}
We begin with the ``if'' implication. Given a finite subset $K\subseteq \Lambda$ and $\varepsilon>0$, let $n\in\mathbb{N}$ and $\varphi\colon \Lambda/\Lambda_n\to\Lambda$ be as in the statement. Set $F=\varphi(\Lambda/\Lambda_n)$, and observe that $|F|=|\Lambda/\Lambda_n|$. Given $\lambda\in \Lambda$ and $t\in \Lambda/\Lambda_n$, we have $\pi_n(\lambda\varphi(t))=\pi_n(\lambda)t$. Using condition~(1), we see that $\lambda\varphi(t)$ belongs to $F$ if and only if $\lambda\varphi(t)=\varphi(\pi_n(\lambda)t)$, and thus
\[
\lambda F\cap F=\lambda\varphi\Big(\big\{t\in \Lambda/\Lambda_n\colon \lambda\varphi(t) = \varphi(\pi_n(\lambda)t)\big\}\Big).
\]
Fix $\lambda\in K$. Using condition~(2) and the equality above, we deduce that $|\lambda F\cap F|\geq (1-\varepsilon)|F|$. Thus, for $\lambda\in K$, we have
\[
\frac{|\lambda F\triangle F|}{|F|}=2\Big(1-\frac{|\lambda F\cap F|}{|F|}\Big)<2\varepsilon.
\] This implies that $\max_{\lambda\in K}\frac{|\lambda F\triangle F|}{|F|}<2\varepsilon$, and so $\Lambda$ is amenable.
		
For the ``only if'' implication, assume that $\Lambda$ is amenable. Let $K\subseteq \Lambda$ be a finite subset, and let $\varepsilon>0$. The conclusion of the discussion in \cite[Section~2]{Wei_monotileable_2001} gives us the existence of $n\in\mathbb{N}$ and a section $\varphi \colon \Lambda/\Lambda_n\to \Lambda$ such that $\varphi(\Lambda/\Lambda_n)$ is $(K,\frac{\varepsilon}{|K|})$-invariant. Set $F=\varphi(\Lambda/\Lambda_n)$ and 
\[
E=\{t \in \Lambda/\Lambda_n \colon \lambda\varphi(t) \neq \varphi(\pi_n(\lambda)t) \ \text{for some} \ \lambda \in K\}.
\]
Note that $|KF\setminus F |<\varepsilon\tfrac{|F|}{|K|}$ by invariance, and that  $|F|=|\Lambda/\Lambda_n|$ since $\varphi$ is injective. It remains to prove that $|E|< \varepsilon \abs{\Lambda/\Lambda_n}$. It suffices to show that there is an injective map
\[
\rho \colon E\to K\times \big( KF\setminus F\big),
\]
since the codomain has cardinality at most $\varepsilon|\Lambda/\Lambda_n|$.

We proceed to construct such a map. Let $t\in E$ and choose $\lambda\in K$ with $\lambda\varphi(t)\neq \varphi(\pi_n(\lambda)t)$. We set $\rho(t)=(\lambda,\lambda\varphi(t))$. To show that $\rho$ is well-defined, we need to show that $\lambda\varphi(t)$ does not belong to $F=\varphi(\Lambda/\Lambda_n)$. Arguing by contradiction, suppose that there exists $s\in \Lambda/\Lambda_n$ with $\lambda\varphi(t)=\varphi(s)$. Applying $\pi_n$ to this equality gives $\pi_n(\lambda)t=s$, and applying $\varphi$ again gives $\varphi(\pi_n(\lambda)t)=\varphi(s)=\lambda\varphi(t)$, which is the desired contradiction.

Finally, we claim that $\rho$ is injective. Let $s,t\in E$ satisfy $\rho(s)=\rho(t)$. Thus there exists $\lambda\in K$ such that $\lambda\varphi(s)= \lambda\varphi(t)$. Hence $\varphi(s)=\varphi(t)$ and since $\varphi$ is injective, we conclude that $s=t$. This finishes the proof.
\end{proof}

Next, we outline the construction of the first examples of allosteric actions of amenable groups due to Joseph; see \cite{Jos23}. Recall that for a prime number $p$, a countable group $\Lambda$ is said to be \emph{residually $p$-finite} if there exists a sequence $(\Lambda_n)_{n\in\mathbb{N}}$ of normal subgroups of $\Lambda$ such that $\bigcap_{n\in\mathbb{N}}\Lambda_n=\{1\}$ and $\Lambda/\Lambda_n$ is a finite group whose cardinality is a power of $p$. 
	
\begin{constr}\label{con Matthieu's examples}{\normalfont{\cite[Theorem~3.2 and Theorem~4.1]{Jos23}}}
Let $\Lambda$ be a countable group that is residually $p$-finite for infinitely many distinct prime numbers $p$. (For example, we may take $\Lambda=\mathbb{Z}$). Fix $d\in\N$, and identify the direct sum $\bigoplus_{\Lambda}\mathbb{Z}^d$ with the group of finitely supported functions $\Lambda\to\mathbb{Z}^d$, with pointwise addition. Denote by $\beta\colon \Lambda \to\mathrm{Aut}(\bigoplus_{\Lambda}\mathbb{Z}^d)$ the action given by shifting the copies of $\mathbb{Z}^d$, namely 
\[
\beta_{\lambda}(f)(\lambda')=f(\lambda^{-1}\lambda')
\]
for all $f\in \bigoplus_{\Lambda}\mathbb{Z}^d$ and $\lambda,\lambda'\in\Lambda$. Set 
\[
\mathbb{Z}^d\wr\Lambda=\bigg(\bigoplus_{\Lambda}\mathbb{Z}^d\bigg)\rtimes_{\beta}\Lambda,
\] 
which is the \textit{wreath product} of $\mathbb{Z}^d$ by $\Lambda$. In other words, $\mathbb{Z}^d\wr\Lambda$ is the cartesian product $\bigoplus_{\Lambda}\mathbb{Z}^d\times \Lambda$ with operations given by
\[
(f,\lambda)\cdot(f',\lambda')=(f+\beta_{\lambda}( f'),\lambda\lambda')\;\;\text{ and }(f,\lambda)^{-1}=(-\beta_\lambda^{-1}(f),\lambda^{-1})
\]
for all $f,f'\in \bigoplus_{\Lambda}\mathbb{Z}^d$ and $\lambda,\lambda'\in\Lambda$. 
We will from now on abbreviate $\mathbb{Z}^d\wr\Lambda$ to $\Gamma$. For each $\gamma = (g, \delta) \in \Gamma \setminus \{1\}$, choose
\begin{itemize}
\item a prime number $p_\gamma$,

\item a positive number $\varepsilon_\gamma$,

\item a natural number $l_\gamma$,

\item a normal subgroup $\Lambda_\gamma \lhd \Lambda$, and

\item a finite subset $E_\gamma \subseteq \Lambda/\Lambda_\gamma$ containing the trivial coset $\Lambda_\gamma$,
\end{itemize}
such that, with $\pi_\gamma \colon \Lambda \to \Lambda/\Lambda_\gamma$ denoting the canonical quotient map, the following are satisfied:
\begin{enumerate}
\item the primes $p_\gamma$, for $\gamma\in\Gamma\setminus\{1\}$, are pairwise distinct, and $\Lambda$ is residually $p_\gamma$-finite,

\item $g(\Lambda) \cap \ {(p_\gamma\mathbb{Z})}^d = \{0\}$ in $\mathbb{Z}^d$,
			
\item $\prod_{\gamma \neq 1}\ {(1 - \varepsilon_\gamma)} > 0$,

\item $l_\gamma > \abs{\supp(g)}$,

\item $[\Lambda : \Lambda_\gamma]$ is a power of $p_\gamma$ and $\varepsilon_\gamma[\Lambda : \Lambda_\gamma] > l_\gamma$,

\item $\supp(g)^{-1}\supp(g) \cap \Lambda_\gamma = \{1\}$,

\item if $\delta \neq 1_\Lambda$ then $\delta \not \in \Lambda_\gamma$, and

\item $\pi_\gamma(\supp(g)) \subseteq E_\gamma$ and $\abs{E_\gamma}$ = $l_\gamma$.
\end{enumerate}

For $\gamma\in\Gamma\setminus\{1\}$ and upon identifying an element $q\in E_{\gamma}\subseteq \Lambda/\Lambda_{\gamma}$ with the corresponding coset in $\Lambda$, set 
\[
A_\gamma = \Big\{h \in \bigoplus_{\Lambda}\mathbb{Z}^d \colon \sum_{\lambda\in q} h(\lambda) \in \ {(p_\gamma\mathbb{Z})}^d \mbox{ for all }  q\in E_{\gamma}\Big\}.
\]
It is shown in \cite[Theorem~4.2, Claim~1]{Jos23} that $A_{\gamma}$ is a subgroup of $\bigoplus_\Lambda \mathbb{Z}^d$, which is invariant under the restriction of $\beta$ to $\Lambda_{\gamma}$. Thus, setting $\Gamma_\gamma = A_\gamma \rtimes \Lambda_\gamma$, it follows that $\Gamma_{\gamma}$ is a subgroup of $\Gamma$, and its index is a (finite) power of $p_\gamma$, see \cite[Theorem~4.2, Claim~2]{Jos23}. For a finite subset $F\subseteq \Gamma\setminus\{1\}$, set $X_F=\Gamma / \bigcap_{\gamma \in F} \Gamma_\gamma$. Then $X_F$ is finite, and the action of $\Gamma$ on itself by left translation induces an action on $X_F$.

Set 
\[
X= \varprojlim_{F \Subset \Gamma\setminus\{1\}} X_F,
\]
which is a Cantor space. Since for $F\subseteq F'$, the canonical quotient map $X_{F'}\to X_{F}$ is $\Gamma$-equivariant, we obtain an inverse limit action $\Gamma\curvearrowright X$. For a finite subset $F\subseteq \Gamma\setminus\{1\}$, we write $q_F\colon X\to X_F$ for the canonical quotient map, and note that $q_F$ is $\Gamma$-equivariant. By \cite[Theorem~3.2 and Theorem~4.1]{Jos23}, the action of $\Gamma$ on $X$ is allosteric.

\end{constr}	
	
We will need convenient notation for the generators of $\Gamma$, which we now establish. We will regard both $\bigoplus_{\Lambda}\mathbb{Z}^d$ and $\Lambda$ as subgroups of $\Gamma$ in a canonical way.

\begin{nota}\label{nota:generators}
For $j=1,\ldots,d$, we set $e_j=(0, \ldots, 1, \ldots, 0)\in \mathbb{Z}^d$, where $1$ is in the $j$-th coordinate. Given $j=1,\ldots,d$ and $\lambda \in \Lambda$, we write $\xi^\lambda_{j}\in \bigoplus_{\Lambda}\mathbb{Z}^d\subseteq \Gamma$ for the function given by $\xi^\lambda_j(\lambda')=\delta_{\lambda,\lambda'}e_j$ for $\lambda'\in\Lambda$. Note that $\xi_j^\lambda=\beta_{\lambda}(\xi_j^{1})$ and $\{\xi_j^\lambda\colon j=1,\ldots,d,\lambda\in\Lambda\}$ generates the copy of $\bigoplus_{\Lambda}\mathbb{Z}^d$ in $\Gamma$. Moreover, if $\mathcal{L}\subseteq \Lambda$ is any generating set, then $\Gamma$ is generated by 
\[
\mathcal{L}\cup \big\{\xi_j^1\colon j=1,\ldots,d\big\}.
\]
For further use, we note here that
\[
\label{eqn:IdentityGamma}\tag{5.1}
\lambda\cdot (n\xi_j^{\lambda'})= n\xi_{j}^{\lambda\lambda'}\cdot\lambda 
\]
for all $\lambda,\lambda'\in\Lambda$, all $j=1,\ldots,d$, and $n\in \mathbb{Z}$.
\end{nota}

\begin{cor}\label{cor: cor to con and section lemma}
Adopt the assumptions and notations introduced in \autoref{con Matthieu's examples}, and in particular set $\Gamma= \mathbb{Z}^d\wr \Lambda$. Assume that $\Lambda$ is amenable. Let $F\subseteq\Gamma$ and $\Lambda_0\subseteq\Lambda$ be finite subsets, let $n\in\mathbb{N}$ and let $\eta>0$. Then there exists a finite subset $E\subseteq\Gamma\setminus\{1\}$ such that, with $\Lambda_E=\bigcap_{\gamma\in E}\Lambda_{\gamma}$ and $\pi\colon \Lambda\to \Lambda/\Lambda_E$ the quotient map, we have
\begin{enumerate}
\item $E\cap F=\emptyset,$

\item there exists a map $\varphi\colon\Lambda/\Lambda_E\to\Lambda$ satisfying $\pi\circ\varphi=\mathrm{id}_{\Lambda/\Lambda_E}$ and
\[
\abs{\bigg\{t\in\Lambda/\Lambda_E\colon \lambda\varphi(t)\ne\varphi(\lambda t)\text{ for some }\lambda\in\Lambda_0\mbox{ or }\varphi(t)\in\bigcup_{\gamma\in E}\Lambda_\gamma\bigg\}}<\eta\cdot|\Lambda/\Lambda_E|,
\]
\item $p_\gamma>n$ for all $\gamma\in E$.
\end{enumerate}
\end{cor}

\begin{proof}
Note that, by condition (5) of \autoref{con Matthieu's examples} we have $\frac{1}{[\Lambda:\Lambda_\gamma]}<\varepsilon_\gamma$ and by condition (3) we have that $\sum_{\gamma\ne 1}\frac{1}{[\Lambda:\Lambda_\gamma]}<\infty$. Denote by $\mathcal{F}(\Gamma)$ the (countable) collection of all finite subsets $E$ of $\Gamma\setminus \{1\}$ such that $E \cap F = \emptyset$, $p_\gamma > n$ for all $\gamma \in E$ and $\sum_{\gamma\in E}\frac{1}{[\Lambda:\Lambda_\gamma]}<\eta/2$. Then $(\Lambda_E)_{E\in \mathcal{F}(\Gamma)}$ is a countable collection of finite index normal subgroups of $\Lambda$.

Let us show that for every finite set $\Theta \subseteq \Lambda$ there exists $H\in\mathcal{F}(\Gamma)$ such that the quotient map $\pi_H \colon \Lambda \to \Lambda/\Lambda_H$ is injective on $\Theta$, in order to apply \autoref{AlmostEquivariantLiftLemma}. Since $\sum_{\gamma\ne 1}\frac{1}{[\Lambda:\Lambda_\gamma]}<\infty$, there is a finite set $K\subseteq\Gamma\setminus\{1\}$ so that if $K'\subseteq\Gamma\setminus\{1\}$ satisfies $K\cap K'=\emptyset$, then $\sum_{\gamma\in K'}\frac{1}{[\Lambda:\Lambda_\gamma]}<\eta/2$. Let $\theta_1, \theta_2 \in \Theta$ with $\theta_1\neq\theta_2$ be given, and set $\theta = \theta_1^{-1}\theta_2 \in \Theta^{-1}\Theta \setminus \{1\}$. Since the set $K\cup F\cup\{\gamma \in \Gamma\colon p_\gamma \le n\}$ is finite, there exists $h_\theta \in \bigoplus_\Lambda \mathbb{Z}^d$ such that $\gamma_\theta\coloneqq (h_\theta, \theta) \notin K\cup F$ and $p_{\gamma_\theta} > n$. By property (7) in \autoref{con Matthieu's examples}, we have $\pi_{\gamma_\theta}(\theta) \ne 1_{\Lambda/\Lambda_{\gamma_\theta}}$ and therefore $\pi_{\gamma_\theta}(\theta_1) \ne \pi_{\gamma_\theta}(\theta_2)$. Hence putting $H\coloneqq \{\gamma_\theta\colon \theta\in \Theta^{-1}\Theta\setminus\{1\}\}$ yields a finite set in $\mathcal{F}(\Gamma)$ such that $\pi_H$ is injective on $\Theta$. By \autoref{AlmostEquivariantLiftLemma}, there exists some $E\in\mathcal{F}(\Gamma)$ and a map $\varphi\colon\Lambda/\Lambda_E\to \Lambda$ such that $\pi_E\circ\varphi=\mathrm{id}_{\Lambda/\Lambda_E}$ and
\[
\big|\big\{t \in \Lambda/\Lambda_E \colon \lambda\varphi(t) \neq \varphi(\pi_n(\lambda)t) \ \text{for some} \ \lambda \in \Lambda_0\big\}\big| < \frac{\eta}{2} \abs{\Lambda/\Lambda_E}.
\]
Moreover, by definition of $\mathcal{F}(\Gamma)$ we have that $\sum_{\gamma\in E}\frac{1}{[\Lambda:\Lambda_\gamma]}<\eta/2$, therefore
\[
\abs{\bigg\{t\in\Lambda/\Lambda_E\colon \varphi(t)\in\bigcup_{\gamma\in E}\Lambda_\gamma\bigg\}}\le\sum_{\gamma\in E}[\Lambda_\gamma:\Lambda_E]=\sum_{\gamma\in E}\frac{[\Lambda:\Lambda_E]}{[\Lambda:\Lambda_\gamma]}<\frac{\eta}{2}|\Lambda/\Lambda_E|
\]
and thus (2) is verified.
\end{proof}

For positive elements $a,b$ in a $C^*$-algebra $A$, we say that $a$ is \emph{Cuntz subequivalent} to $b$ (in $A$), written $a\precsim b$, if there exists a sequence $(d_n)_{n\in\mathbb{N}}$ in $A$ such that $\lim_{n\to \infty}\|a-d_nbd_n^*\|=0$. (We refer the reader to either \cite[Chapter 2]{AntPerThi_tensor_2018} or \cite{GarPer_modern_2022} for surveys on Cuntz comparison and the Cuntz semigroup.)

We recall the following definition from \cite{HirOro13}.

\begin{df}\label{df:trZst}(\cite[Definition~2.1]{HirOro13}).
Let $A$ be a unital $C^*$-algebra. We say that $A$ is \emph{tracially $\mathcal{Z}$-stable} if $A\neq \mathbb{C}$ and for every finite subset $S\subseteq A$, every $\varepsilon>0$, every $a\in A_+\setminus\{0\}$ and every $n\in\mathbb{N}$, there exists a completely positive contractive order zero map $\rho\colon M_n\to A$ such that
\begin{enumerate} 
\item $1-\rho(1)\precsim a$, and 

\item for any $x\in M_n$ and $b\in S$ we have $\|\rho(x)b-b\rho(x)\|\leq \varepsilon\|x\|$.
\end{enumerate}
\end{df}

Given $n\in\mathbb{N}$, let $\{e_{i,j}\}_{i,j=1}^n$ denote a system of matrix units for $M_n$. Upon replacing $\varepsilon$ by $\varepsilon/n^2$ in the definition above, condition (2) can be equivalently replaced with $\|\rho(e_{i,j})b-b\rho(e_{i,j})\|\le\varepsilon$ for all $b\in S$ and all $i,j=1,\dots,n$. We will use this at the end of the proof of \autoref{thm allosteric actions Z-stable}.

In practice, tracial $\mathcal{Z}$-stability is easier to establish than $\mathcal{Z}$-stability. For simple, separable, unital, nuclear $C^*$-algebras, it is shown in \cite[Theorem~4.1]{HirOro13} that tracial $\mathcal{Z}$-stability is in fact equivalent to $\mathcal{Z}$-stability, making it a useful criterion to apply in concrete examples.

For a positive function $a\in C(X)$ on a compact Hausdorff space $X$, we write
$\mathrm{supp}(a)$ for its open support, namely the open set $a^{-1}((0,\infty))$. The following is implicit in the literature (it can be deduced, for example, from the arguments in \cite[Lemma 12.3]{Ker_dimension_2020}), and we isolate it here for later use.

\begin{prop}\label{prop:DynCompCuntz}
Let $\Gamma$ be a discrete group, let $X$ be a compact Hausdorff space, and let $\Gamma\curvearrowright X$ be an action with dynamical comparison. Let $a,b\in C(X)$ be positive functions satisfying 
\[
\mu\big(\mathrm{supp}(a)\big)<\mu\big(\mathrm{supp}(b)\big)
\]
for all $\Gamma$-invariant Borel probability measures $\mu$ on $X$. Then $a\precsim b$ in $C(X)\rtimes_{\mathrm{r}}\Gamma$.
\end{prop}

We now come to the main result of this section: the crossed products of Joseph's actions are $\mathcal{Z}$-stable. In particular, since they are also simple, separable, unital, nuclear, and satisfy the UCT, these $C^*$-algebras are classifiable.

\begin{thm}\label{thm allosteric actions Z-stable}
Let $\Lambda$ be a countable amenable group that is residually $p$-finite for infinitely many distinct prime numbers $p$, and let $d\in\N$. Set $\Gamma=\mathbb{Z}^d\wr\Lambda$ and let $\alpha$ denote an allosteric action of $\Gamma$ on the Cantor space $X$ obtained following the procedure described in \autoref{con Matthieu's examples}. Then, the crossed product $C(X)\rtimes_{\alpha}\Gamma$ is $\mathcal{Z}$-stable. 
\end{thm}

\begin{proof}
The crossed product is clearly unital, separable and nuclear. Since $\alpha$ is minimal and topologically free, the work of Archbold--Spielberg \cite{ArcSpi94} implies that $C(X)\rtimes_{\alpha}\Gamma$ is simple. Thus, by \cite[Theorem 4.1]{HirOro13}, in order to establish $\mathcal{Z}$-stability for $C(X)\rtimes_{\alpha}\Gamma$, it suffices to show tracial $\mathcal{Z}$-stability. Throughout the proof, we will identify a group element $\gamma\in\Gamma$ with the canonical unitary $u_{\gamma}\in C(X)\rtimes_{\alpha}\Gamma$, and will thus, for example, write $\Gamma\subseteq C(X)\rtimes_{\alpha}\Gamma$.

Let $S\subseteq C(X)\rtimes_{\alpha}\Gamma$ be a finite subset, let $\varepsilon>0$, let $a\in C(X)\rtimes_{\alpha}\Gamma$ be a nonzero positive element, and let $n\in\mathbb{N}$ be given. We begin by making a number of standard reductions and simplifications in our setting. First, by \cite[Lemma~7.9]{Phi14} we can assume without loss of generality that $a\in C(X)_+\setminus\{0\}$. Also, we may assume that $\varepsilon<\|a\|/2$. Using that $C(X)$ and $\Gamma$ generate $C(X)\rtimes_{\alpha}\Gamma$, that $C(X)$ is the inductive limit of $C(X_F)$ for finite subsets $F\subseteq \Gamma\setminus\{1\}$, and adopting the notation introduced in \autoref{nota:generators}, we will without loss of generality assume that there exist finite subsets $F\subseteq \Gamma\setminus\{1\}$, $S_0\subseteq C(X_F)$ and $\Lambda_0\subseteq \Lambda$, and a positive element $b\in C(X_F)$ such that 
\[
S=S_0\cup \Lambda_0 \cup \big\{\xi_k^1\colon k=1,\ldots,d\big\} \ \ \mbox{ and } \ \ \|a-b\|<\varepsilon.
\]
By \cite[Lemma~2.5]{GarPer_modern_2022}, we have $(b-\varepsilon)_+\precsim a$. Since $(b-\varepsilon)_+\in C(X_F)$ is positive and nonzero (the latter because $\varepsilon<\|a\|/2$), upon replacing $a$ with it we may assume that $a\in C(X_F)$. It is proved in \cite[Lemma~3.1]{Jos23} that the action $\Gamma\curvearrowright X$ is uniquely ergodic, and we denote by $\mu$ the unique $\Gamma$-invariant Borel probability measure on $X$. Upon replacing $\varepsilon$ with $\min\big\{\varepsilon, \tfrac{1}{2}\mu\big(\mathrm{supp}(a)\big)\big\}$, we will assume without loss of generality that 
\[
\label{eq:musuppbig}\tag{5.2}\mu\big(\mathrm{supp}(a)\big)>\varepsilon.
\]
Finally, upon replacing $\Lambda_0$ with $\Lambda_0\cup \Lambda_0^{-1}$, we will assume that $\Lambda_0$ is symmetric.

Fix a natural number $m\in\mathbb{N}$ with $m > 2/\varepsilon$, and let $\eta>0$ satisfy
\[
\eta<\frac{\varepsilon}{2\big(1 + \abs{\Lambda_0} + \abs{\Lambda_0}^2 + \ldots + \abs{\Lambda_0}^{m-1}\big)}.
\]

Use \autoref{cor: cor to con and section lemma} to find a finite subset $E\subseteq\Gamma\setminus\{1\}$ with $E\cap F=\emptyset$, such that $p_\gamma>2n/\varepsilon$ for all $\gamma\in E$ and such that there exists a map $\varphi\colon\Lambda/\Lambda_E\to\Lambda$ satisfying
\begin{enumerate}
\item[(a)] $\pi \circ \varphi = \mathrm{id}_{\Lambda/\Lambda_E}$, where $\pi\colon \Lambda\to \Lambda/\Lambda_E$ is the canonical quotient map, and

\item[(b)] we have
\[
\abs{\bigg\{t \in \Lambda/\Lambda_{E} \colon \lambda\varphi(t) \neq \varphi(\lambda t) \ \text{for some} \ \lambda \in \Lambda_0\mbox{ or } \varphi(t)\in\bigcup_{\gamma\in E}\Lambda_\gamma\bigg\}} < \eta \abs{\Lambda/\Lambda_E}.
\]
\end{enumerate}

With $Y=\Lambda/\Lambda_E$ and 
\[
Y_0=\bigg\{t \in \Lambda/\Lambda_{E} \colon \lambda\varphi(t) \neq \varphi(\lambda t) \ \text{for some} \ \lambda \in \Lambda_0\mbox{ or } \varphi(t)\in\bigcup_{\gamma\in E}\Lambda_\gamma\bigg\},
\] 
use \autoref{WeddingCakeLemmaFinite} to find a function $f\colon \Lambda/\Lambda_E\to [0,1]$ satisfying
\begin{enumerate}
\item[(i)] $f(t) = 0$ for all $t \in Y_0$,

\item[(ii)] $\abs{f(\lambda t) - f(t)} \le 1/m$ for all $t\in Y$ and all $\lambda \in \Lambda_0$, and

\item[(iii)] $\abs{\left\{t \in \Lambda/\Lambda_E \colon f(t) \neq 1\right\}} \le \big(1 + \abs{\Lambda_0} + \abs{\Lambda_0}^2 + \ldots + \abs{\Lambda_0}^{m-1}\big)\abs{Y_0}$.
\end{enumerate}
In particular, we get 
\begin{enumerate}
\item[(c)] $\lambda\varphi(t) = \varphi(\lambda t)$ for all $\lambda \in \Lambda_0$ and all $t\notin Y_0$.
\end{enumerate}
Moreover, 
\[\label{eqn:SmallMeas}\tag{5.3}
\frac{\abs{\left\{t \in \Lambda/\Lambda_E \colon f(t) \neq 1\right\}}}{|Y|}
\ \stackrel{\mathrm{(iii)}}{\leq} \ \frac{\big(1 + \abs{\Lambda_0} + \abs{\Lambda_0}^2 + \ldots + \abs{\Lambda_0}^{m-1}\big)\abs{Y_0}}{|Y|}\ \stackrel{\mathrm{(b)}}{<} \ \frac{\varepsilon}{2}.
\]
		
As in \autoref{con Matthieu's examples}, for any $\gamma\in\Gamma\setminus\{1_\Gamma\}$ we set 
\[
A_\gamma \coloneqq \Big\{\xi \in \bigoplus_{\Lambda}\mathbb{Z}^d \colon \sum_{\lambda\in q} \xi(\lambda) \in \ {(p_\gamma\mathbb{Z})}^d \mbox{ for all }  q\in E_{\gamma}\Big\},
\]
which is a subgroup of $\bigoplus_{\Lambda}\mathbb{Z}^d$ that is invariant under the restriction of $\beta$ to $\Lambda_{\gamma}$. Recall that we set $\Gamma_{\gamma}\coloneqq A_{\gamma}\rtimes \Lambda_{\gamma}$ and 
		$\Gamma_E=\bigcap_{\gamma\in E}\Gamma_\gamma$, which are
		finite index subgroups of $\Gamma$, and consider
		the finite $\Gamma$-space
		$X_E=\Gamma/\Gamma_E$.
		
		Set $P = \prod_{\kappa \in F}p_\kappa$. Given $\xi \in \bigoplus_{\Lambda}\mathbb{Z}^d$, since $p_\kappa \xi\cdot \Gamma_{\kappa}\subseteq \Gamma_{\kappa}$, one checks that $p_\kappa \xi\cdot \gamma\Gamma_\kappa\subseteq \gamma\Gamma_\kappa$ for all $\gamma\in\Gamma$, and therefore $P\xi$ acts trivially
		on $X_{\kappa}$ for all $\kappa\in F$. In particular, $P\xi$ acts trivially on $X_F$. Denoting by $\alpha^{F}$ the action of $\Gamma$ on $X_F$, we have
\[\label{eq:5.1}\tag{5.4}
\alpha^F_{P\xi}(h)=h
\]
for all $h\in C(X_F)$. 
Given $x\in \mathbb{Z}^d$, we write $x_1\in\mathbb{Z}$ for its first coordinate. Set $Q\coloneqq\prod_{\gamma\in E}p_\gamma$ and define 
\[
W\coloneqq \Big\{(\xi, 1_\Lambda)\Gamma_E \in X_E \colon \xi\in \bigoplus_{\Lambda}\mathbb{Z}^d \mbox{ and } \sum_{\lambda \in \Lambda_\gamma}\xi(\lambda)_1 \equiv 0 \ \mathrm{mod} \ p_\gamma \mbox{ for all }\gamma\in E\Big\},
\]
which is a subset of the finite set $X_{E}$. 

Since $E\cap F=\emptyset$, condition~(1) in \autoref{con Matthieu's examples} implies that $Q$ and $P$ are coprime. 
Recall that in \autoref{con Matthieu's examples}, all the sets $E_\gamma$ are assumed to contain the identity $\Lambda_\gamma \in \Lambda / \Lambda_\gamma$, and thus
the trivial coset $\Lambda_\gamma$ belongs to $E_\gamma$ for all $\gamma\in \Gamma\setminus\{1\}$. Using this, we deduce that the sets 
\begin{align*}
& \ \ \ \ \ \ \ \ \ \ \ \ \ (jP\xi_1^{1_\Lambda})W \\
&= \Big\{(\xi, 1_\Lambda)\Gamma_E \in X_E \colon \xi\in \bigoplus_{\Lambda}\mathbb{Z}^d \mbox{ and }  \sum_{\lambda \in \Lambda_\gamma}\xi(\lambda)_1 \equiv jP \ \mathrm{mod} \ p_\gamma \mbox{ for all }\gamma\in E  \Big\},
\end{align*}
are pairwise disjoint for $j=0,\dots, Q-1$, and, by the Chinese Remainder Theorem, their union is 
\[
\bigsqcup_{j=0}^{Q-1}(jP\xi_1^{1_\Lambda})W = \Big\{(\xi, 1_\Lambda)\Gamma_E \in X_E \colon \xi\in \bigoplus_{\Lambda}\mathbb{Z}^d\Big\}.
\]
	
A routine calculation shows that, for every $\lambda\in \Lambda$, one has
\begin{align*}
& \ \ \ \ \ \ \ \ \ \ \ \ \ jP\xi_1^\lambda\lambda W=\lambda\cdot(jP\xi_1^{1_\Lambda})W\\
& = \Big\{(\xi,\lambda)\Gamma_E\in X_E\colon\xi\in\bigoplus_\Lambda\mathbb{Z}^d,\mbox{ and }\sum_{\lambda^\prime \in \lambda\Lambda_\gamma}\xi(\lambda^\prime)_1\equiv jP \ \mathrm{mod} \ p_\gamma \mbox{ for all }\gamma\in E\Big\}.
\end{align*}
Thus, we have
\[
\bigsqcup_{j=0}^{Q-1}(jP\xi_1^{\lambda})(\lambda W) = \Big\{(\xi, \lambda)\Gamma_E \in X_E \colon \xi\in \bigoplus_{\Lambda}\mathbb{Z}^d\Big\},
\]
for all $\lambda\in\Lambda$, and it follows that $X_E$ can be written as the disjoint union
\[\label{eqn:PartXgamma}\tag{5.5}
X_E = \bigsqcup_{t \in\Lambda/\Lambda_E}\bigsqcup_{j=0}^{Q-1}\big(jP\xi_1^{\varphi(t)}\big)(\varphi(t) W).
\]
Moreover, given $k=1,\ldots,d$ and $t\in \Lambda/\Lambda_{E}$ with $\varphi(t)\not\in\bigcup_{\gamma\in E} \Lambda_\gamma$, observe that $1_\Lambda\not\in\varphi(t)\Lambda_\gamma$ for any $\gamma\in E$, and therefore
\[
\sum_{\lambda\in\varphi(t)\Lambda_\gamma}\xi_k^1(\lambda)_1=0,
\]
whence
\[\label{eqn:xiWW}\tag{5.6}
\xi_k^1\varphi(t)W = \varphi(t)W.
\]

Dividing $Q$ by $n$, find $c\ge1$ and $r\in\{0,\dots,n-1\}$ so that $Q=nc+r$. For $j=0,\dots,n-1$, define 
\[
R_j= \big\{j,j+n,j+2n,\dots,j+(c-1)n\big\}\subseteq\mathbb{N}.
\] 
Note that $|R_j|=c$ and that $\bigsqcup_{j=0}^{n-1}R_j=\{0,\dots, nc-1\}$. Now for each $t\in\Lambda/\Lambda_E$ and each $j=0,\dots,n-1$, set 
\[\label{eqn:XEJT}\tag{5.7}
X_{E,j,t}\coloneqq \bigsqcup_{l\in R_j}lP\xi_1^{\varphi(t)}\varphi(t)W
\]
and
\[
X_R\coloneqq\bigsqcup_{t\in\Lambda/\Lambda_E}\bigsqcup_{l=nc}^{Q-1}lP\xi_1^{\varphi(t)}\varphi(t)W,
\]
with the convention that $X_R=\emptyset$ if $Q=nc$. By \eqref{eqn:PartXgamma} we obtain
\[\label{eqn:decomp}\tag{5.8}
X_E=X_R \sqcup \bigsqcup_{t\in\Lambda/\Lambda_E}\bigsqcup_{j=0}^{n-1}X_{E,j,t}.
\]

We now proceed to define a linear map $\rho\colon M_{n}(\mathbb{C})\to C(X_E) \rtimes \Gamma$ which, once composed with the inclusion $C(X_E) \rtimes \Gamma\hookrightarrow C(X) \rtimes \Gamma$, will be shown to satisfy the conditions of \autoref{df:trZst}. Let $\{e_{i,j}\}_{i,j=0}^{n-1}\subseteq M_{n}$ denote a system of matrix units for $M_n$, which we index by $i,j=0,\dots,n-1$ for convenience. Define
\[
\rho(e_{i,j})=\sum_{t \in \Lambda/\Lambda_E} f(t)\mathbbm{1}_{X_{E,i,t}}u_{P(i-j)\xi_1^{\varphi(t)}} \in C(X_E) \rtimes \Gamma
\]
and extend linearly to $M_n$. We claim that $\rho$ is completely positive contractive and order zero. To see this, first observe that
\[
\rho(1)=\sum_{t\in\Lambda/\Lambda_E}f(t)\mathbbm{1}_{\bigsqcup_{j=0}^{n-1}X_{E,j,t}}\in C(X_E)
\]
is positive and contractive. Define a map $\psi\colon M_n\to C(X_E)\rtimes\Gamma$ by setting
\[
\psi(e_{i,j})\coloneqq\sum_{t\in\Lambda/\Lambda_E}\mathbbm{1}_{X_{E,i,t}}u_{P(i-j)\xi_1^{\varphi(t)}}
\]
for all $i,j=0,\dots,n-1$ and extending linearly. Some tedious but nevertheless routine calculations --the details of which we omit-- then verify that $\psi$ is a $*$-homomorphism and that $\rho(e_{i,j})=\rho(1)\psi(e_{i,j})=\psi(e_{i,j})\rho(1)$ for all $i,j=0,\dots,n-1$. We conclude that $\rho$ is completely positive, contractive, and order zero. 

It remains to show that $\rho$ satisfies the conditions in \autoref{df:trZst}. In order to prove that $1-\rho(1)\precsim a$, note first that
\[
1-\rho(1)=\mathbbm{1}_{X_R}+\sum_{t\in\Lambda/\Lambda_E}(1-f(t))\mathbbm{1}_{\bigsqcup_{j=0}^{n-1}X_{E,j,t}}.
\]
It is proved in \cite[Lemma~3.1]{Jos23} that the action $\Gamma\curvearrowright X$ is uniquely ergodic, and that the unique $\Gamma$-invariant probability measure $\mu$ is induced by the normalized counting measures on each building block $X_F$.
Using at the fifth step that $n/Q<\varepsilon/2$ since $p_\gamma> 2n/\varepsilon$ for all $\gamma\in E$, we deduce that
\begin{align*}
\mu\big(\mathrm{supp}(1-\rho(1))\big)&=\mu(X_R)+\mu\Big(\bigsqcup_{\substack{t\in\Lambda/\Lambda_E\\ f(t)\ne 1}}\bigsqcup_{j=0}^{n-1}X_{E,j,t}\Big) \\
&=\frac{|\Lambda/\Lambda_E|\cdot r\cdot|W|}{|\Lambda/\Lambda_E|\cdot Q\cdot |W|}+\frac{|\{t\in\Lambda/\Lambda_E\colon f(t)\ne 1\}|\cdot n\cdot c\cdot |W|}{|\Lambda/\Lambda_E|\cdot Q\cdot |W|}\\
&\le \frac{n}{Q}+\frac{|\{t\in\Lambda/\Lambda_E\colon f(t)\ne 1\}|}{|\Lambda/\Lambda_E|}\\
&\stackrel{\mathmakebox[\widthof{=}]{\eqref{eqn:SmallMeas}}}{<}\frac{n}{Q}+\frac{\varepsilon}{2}\\
&<\varepsilon\\
&\stackrel{\mathmakebox[\widthof{=}]{\eqref{eq:musuppbig}}}{\leq} \mu(\mathrm{supp}(a)).
\end{align*}

By \cite[Lemma~2.4]{Jos23}, the profinite action $\Gamma\curvearrowright X$ has dynamical comparison, and thus by \autoref{prop:DynCompCuntz} we conclude that $1-\rho(1)\precsim a$. 

We proceed to verify condition~(2) in \autoref{df:trZst}. Recall that we have assumed that $S$ has the form $S=S_0\cup \Lambda_0 \cup \big\{\xi_k^1\colon k=1,\ldots,d\big\}$.
Given $h \in S_0\subseteq  C(X_F)$ and $0\leq i,j\leq n-1$, we have
\begin{align*}
h\rho(e_{i,j})&=h\sum_{t \in \Lambda/\Lambda_E} f(t)\mathbbm{1}_{X_{E,i,t}}u_{P(i-j)\xi_1^{\varphi(t)}} \\
&= \sum_{t \in \Lambda/\Lambda_E} f(t)\mathbbm{1}_{X_{E,i,t}}u_{P(i-j)\xi_1^{\varphi(t)}}\alpha^F_{(j-i)P\xi_1^{\varphi(t)}}(h) \\
&\stackrel{\mathmakebox[\widthof{=}]{{\eqref{eq:5.1}}}}{=} \sum_{t \in \Lambda/\Lambda_E} f(t)\mathbbm{1}_{X_{E,i,t}}u_{P(i-j)\xi_1^{\varphi(t)}}h\\
&=\rho(e_{i,j})h.
\end{align*}
Next, let $k = 1, \ldots, d$ and $0\leq i,j\leq n-1$ be given. Note that, in the group $\Gamma$, we have 
$\xi_k^1 \cdot \ell P\xi_1^{\varphi(t)}= \ell P \xi_1^{\varphi(t)} \cdot\xi_k^1$ for all $t\in\Lambda/\Lambda_E$ and all $\ell\in\mathbb{Z}$, since the factors belong to the (commutative) subgroup
$\bigoplus_{\Lambda}\mathbb{Z}^d$. Using this, we see that for $t \in \Lambda/\Lambda_E$ such that $\varphi(t)\not\in\bigcup_{\gamma\in E}\Lambda_\gamma$, one has
\[\label{eqn:xik1XEJT}\tag{5.9}
\xi_k^1X_{E,i,t}=\bigsqcup_{l\in R_i}lP\xi_1^{\varphi(t)}\xi_k^1\varphi(t)W\stackrel{\eqref{eqn:xiWW}}{=}\bigsqcup_{l\in R_i}lP\xi_1^{\varphi(t)}\varphi(t)W=X_{E,i,t}.
\]
Using commutativity again, as well as the fact that $f(t)=0$ whenever $t\in\Lambda/\Lambda_E$ satisfies $\varphi(t)\in\bigcup_{\gamma\in E}\Lambda_\gamma$, we obtain
\begin{align*}
u_{\xi_k^1}\rho(e_{i,j})&=u_{\xi_k^1} \sum_{t \in \Lambda/\Lambda_E} f(t)\mathbbm{1}_{X_{E,i,t}}u_{P(i-j)\xi_1^{\varphi(t)}} \\
&= \sum_{t \in \Lambda/\Lambda_E} f(t)\mathbbm{1}_{\xi_k^1X_{E,i,t}}u_{P(i-j)\xi_1^{\varphi(t)}}u_{\xi_k^1} \\ 
&\stackrel{\mathmakebox[\widthof{=}]{{\eqref{eqn:xik1XEJT}}}}{=}\sum_{t \in \Lambda/\Lambda_E} f(t)\mathbbm{1}_{X_{E,i,t}}u_{P(i-j)\xi_1^{\varphi(t)}} u_{\xi_k^1} \\
&=\rho(e_{i,j})u_{\xi_k^1}.
\end{align*}

Lastly, for elements $\lambda \in \Lambda_0$ and $0\leq i,j\leq n-1$, using at the seventh and tenth steps that the sets $X_{E,i,t}$, for $t\in\Lambda/\Lambda_E$, are pairwise disjoint (see \eqref{eqn:decomp}), we get
\begin{align*}
u_\lambda \rho(e_{i,j})&=u_\lambda \sum_{t \in \Lambda/\Lambda_E} f(t)\mathbbm{1}_{X_{E,i,t}} u_{P(i-j)\xi_1^{\varphi(t)}} \\
&= \sum_{t \in \Lambda/\Lambda_E} f(t)\mathbbm{1}_{\lambda X_{E,i,t}}u_{\lambda\cdot P(i-j)\xi_1^{\varphi(t)}} \\
&\stackrel{\eqref{eqn:IdentityGamma}}{=} \sum_{t \in \Lambda/\Lambda_E} f(t)\mathbbm{1}_{\lambda X_{E,i,t}}u_{P(i-j)\xi_1^{\lambda\varphi(t)}}u_\lambda\\
&\stackrel{\eqref{eqn:IdentityGamma},\;\eqref{eqn:XEJT}}{=} \sum_{t \in \Lambda/\Lambda_E} f(t)\mathbbm{1}_{\bigsqcup_{l\in R_i}lP\xi_1^{\lambda\varphi(t)}\lambda\varphi(t)W}u_{P(i-j)\xi_1^{\lambda\varphi(t)}}u_\lambda\\ 
&\stackrel{\mathrm{(i)}}{=} \sum_{t \notin Y_0} f( t)\mathbbm{1}_{\bigsqcup_{l\in R_i}lP\xi_1^{\lambda\varphi(t)}\lambda\varphi(t)W}u_{P(i-j)\xi_1^{\lambda\varphi(t)}}u_\lambda \\
&\stackrel{\mathrm{(c)},\;\eqref{eqn:XEJT}}{=} \sum_{t \notin Y_0} f( t)\mathbbm{1}_{X_{E,i,\lambda t}}u_{P(i-j)\xi_1^{\varphi(\lambda t)}}u_\lambda \\
&\stackrel{\mathrm{(ii)}}{\approx}_{\hspace*{-1.5px} \varepsilon/2}\sum_{t \notin Y_0} f( \lambda t)\mathbbm{1}_{X_{E,i,\lambda t}}u_{P(i-j)\xi_1^{\varphi(\lambda t)}}u_\lambda \\
&=\sum_{t \notin  \lambda  Y_0} f( t)\mathbbm{1}_{X_{E,i,t}}u_{P(i-j)\xi_1^{\varphi( t)}}u_\lambda \\
&=\rho(e_{i,j})u_\lambda - \sum_{t \in  \lambda  Y_0} \hspace{-1.2cm} \underbrace{f( t)}_{\hspace{1.4cm} < \, \varepsilon/2 \ \text{by (i) and (ii)}}\hspace{-1.2cm} \mathbbm{1}_{X_{E,i,t}}u_{P(i-j)\xi_1^{\varphi( t)}}u_\lambda \\
&\approx_{ \varepsilon/2}
\rho(e_{i,j})u_\lambda.
\end{align*}
In particular, $u_\lambda\rho(e_{i,j})\approx_\varepsilon\rho(e_{i,j})u_\lambda$. This concludes the proof.
\end{proof}

\end{document}